\documentclass[12pt,leqno]{amsart}
\usepackage{fullpage}
\usepackage{amsmath}
\usepackage{amsfonts}
\usepackage{amssymb}
\usepackage{amsthm}
\usepackage{graphicx,xy}
\usepackage{enumerate}

\DeclareMathAlphabet{\mathpzc}{OT1}{pzc}{m}{it}

\def\bb{\mathbb}
\def\cal{\mathcal}
\def\d{{\mathrm d}}
\def\F2{\mathbb F_2}
\def\A2{{\mathcal A}_2}

\def\Z{{\mathbb Z}}

\def\C{{\mathbb C}}
\def\R{{\mathbb R}}

\def\SC{{\cal{SC}}}
\def\LP{{\cal{LP}}}
\def\OC{{\cal{OC}}}
\def\Lie{{\cal{L}ie}}
\def\Com{{\cal{C}om}}
\def\SCvor{{\cal{SC}}^{\rm vor}}

\def\Hoch{{\rm Hoch}}
\def\Sk{\mathbf{Sk}}
\def\dgvs{{\bf dgvs}}
\def\Der{{\rm Der}}
\def\dim{{\rm dim}}
\def\Fin{{\bf Fin}}

\def\Maps{{\rm Maps}}
\def\Coder{{\rm Coder}}
\def\Maps{{\rm Hom}}
\def\Hom{{\rm Hom}}
\def\hom{{\rm hom}}
\def\End{{\rm End}}
\def\sgn{{\rm sgn}}
\def\ide{{\rm id}}

\def\cl{\mathpzc c}
\def\op{\mathpzc o}

\def\kfield{\mathbf k}

\newcommand{\ac}{\scriptstyle \text{\rm !`}}
\newcommand{\epi}{\twoheadrightarrow}

\newcommand{\wiggly}{\xymatrix@1@C=15pt{  && \ar@{~}[ll]}}

\newcommand{\straight}{\xymatrix@1@C=30pt{  & \ar@{-}[l]}}

\newcommand{\whistle}[2]{\objectmargin={0pt}
\xymatrix@R=18pt{ 
  #1  \\ *-<5pt>{\bullet}\ar@{~}[u]\ar@{#2}[d] \\  \ 
}}

\newcommand{\lie}[3]{\objectmargin={0pt}
\xymatrix@R=16pt{ 
  {}\save[]+<-14pt,4pt>*{#1} +<30pt,0pt>*{#2}  \restore                  \\     
  *-<4pt>{\bullet}\ar@{~}[u]!<-15pt,0pt>\ar@{~}[u]!<15pt,0pt>\ar@{#3}[d] \\  \  
}}

\newcommand{\comm}[3]{\objectmargin={0pt}
\xymatrix@R=16pt{ 
  {}\save[]+<-14pt,4pt>*{#1} +<30pt,0pt>*{#2}  \restore                  \\     
  *-<2.5pt>{\circ}\ar@{~}[u]!<-15pt,0pt>\ar@{~}[u]!<15pt,0pt>\ar@{#3}[d] \\  \  
}}

\newcommand{\ass}[3]{\objectmargin={0pt}
\xymatrix@R=16pt{ 
  {}\save[]+<-15pt,4pt>*{#1} +<30pt,0pt>*{#2}  \restore                  \\     
  *-<4pt>{\bullet}\ar@{-}[u]!<-15pt,0pt>\ar@{-}[u]!<15pt,0pt>\ar@{#3}[d] \\  \  
}}

\newcommand{\act}[3]{\objectmargin={0pt}
\xymatrix@R=16pt{ 
  {}\save[]+<-14pt,4pt>*{#1} +<30pt,0pt>*{#2}  \restore                  \\     
  *-<4pt>{\bullet}\ar@{~}[u]!<-15pt,0pt>\ar@{-}[u]!<15pt,0pt>\ar@{#3}[d] \\  \  
}}

\newcommand{\ntwoone}{
\objectmargin={0pt}
\xymatrix@R=16pt{
 {}\save[]+<-14pt,4pt>*{\txt{\tiny 1}} +<14pt,0pt>*{\txt{\tiny 2}} +<14pt,0pt>*{\txt{\tiny 1}} \restore \\
  *[o]-<5pt>{\bullet}\ar@{~}[u]!<-15pt,0pt>\ar@{~}[u]\ar@{-}[u]!<15pt,0pt>\ar@{-}[d]!<0pt,5pt>         \\ \  
}}

\newcommand{\ntwoonedisplay}{
\objectmargin={0pt}
\xymatrix{
 {}\save[]+<-19pt,4pt>*{\txt{\tiny 1}} +<19pt,0pt>*{\txt{\tiny 2}} +<19pt,0pt>*{\txt{\tiny 1}} \restore \\
  *[o]-<5pt>{\bullet}\ar@{~}[u]!<-20pt,0pt>\ar@{~}[u]\ar@{-}[u]!<20pt,0pt>\ar@{-}[d]!<0pt,5pt>         \\ \  
}}

\newcommand{\closedcorolla}[2]{
\objectmargin={0pt}
\xymatrix@R=16pt{
 {}\save[]+<0pt,8pt>*{#1} +<4pt,-12pt>*{...} \restore       \\
 *[o]-<5pt>{\bullet}\ar@{~}[u]!<-5pt,0pt>\ar@{~}[u]!<-15pt,0pt>
                                                \ar@{~}[u]!<15pt,0pt>\ar@{#2}[d]   \\ \ 
}}

\newcommand{\Closedcorolla}[2]{
\objectmargin={0pt}
\xymatrix{
 #1  \\
 *[o]-<5pt>{\bullet}\ar@{~}[u]!<-10pt,0pt>\ar@{~}[u]!<-30pt,0pt>
                               \ar@{~}[u]!<30pt,0pt>\ar@{#2}[d]!<0pt,8pt>   \\ \ 
}}

\newcommand{\rootedoccorolla}[4]{
\objectmargin={0pt}
\xymatrix{
 {}\save[]+<-16pt,8pt>*{\overbrace{\hspace{30pt}}^{#1}} +<6pt,-10pt>*{...}  
 +<15pt,6pt>*{#2} +<8pt,0pt>*{#3} +<18pt,0pt>*{#4} +<-13pt,-7pt>*{...} \restore \\
  *[o]-<5pt>{\bullet}\ar@{~}[u]!<-30pt,0pt>\ar@{~}[u]!<-20pt,0pt>\ar@{~}[u]!<-5pt,0pt>     
   \ar@{-}[u]!<5pt,0pt>\ar@{-}[u]!<12pt,0pt>\ar@{-}[u]!<30pt,0pt>\ar@{-}[d]!<0pt,10pt> \\ \  
}}

\newcommand{\Rootedoccorolla}[6]{
\objectmargin={0pt}
\xymatrix{
 {}\save[]+<-30pt,4pt>*{#1} +<12pt,0pt>*{#2} +<8pt,-8pt>*{...} +<6pt,8pt>*{#3} 
 +<9pt,0pt>*{#4} +<8pt,0pt>*{#5} +<18pt,0pt>*{#6} +<-13pt,-7pt>*{...} \restore \\
  *[o]-<5pt>{\bullet}\ar@{~}[u]!<-30pt,0pt>\ar@{~}[u]!<-20pt,0pt>\ar@{~}[u]!<-5pt,0pt>     
   \ar@{-}[u]!<5pt,0pt>\ar@{-}[u]!<12pt,0pt>\ar@{-}[u]!<30pt,0pt>\ar@{-}[d]!<0pt,10pt> \\ \  
}}

\newcommand{\openedge}[7]{
\objectmargin={0pt}
\xymatrix{
 {}\save[]+<-22pt,8pt>*{\overbrace{\hspace{38pt}}^{#1}}   
  +<27pt,-4pt>*{#2}+<14pt,0pt>*{#3}+<12pt,0pt>*{#4}+<14pt,0pt>*{#5}+<13pt,-1pt>*{#6}
  +<-70pt,-5pt>*{...} +<23pt,0pt>*{...} +<33pt,0pt>*{...} \restore \\
  *[o]-<5pt>{\bullet}\ar@{~}[u]!<-40pt,0pt>\ar@{~}[u]!<-25pt,0pt>\ar@{~}[u]!<-5pt,0pt>     
   \ar@{-}[u]!<5pt,0pt>\ar@{-}[u]!<20pt,0pt>\ar@{-}[u]!<30pt,0pt>\ar@{-}[u]!<40pt,0pt>\ar@{-}[u]!<55pt,0pt>
   \ar@{#7}[d]!<0pt,10pt> \\ \  
}}

\newcommand{\occorolla}[3]{
\objectmargin={0pt}
\xymatrix{
 {}\save[]+<-12pt,6pt>*{\overbrace{}^{#1}} +<2pt,-10pt>*{...}  
 +<16pt,8pt>*{#2} +<15pt,0pt>*{#3} +<-10pt,-8pt>*{...} \restore \\
  *[o]-<5pt>{\bullet}\ar@{~}[u]!<-20pt,0pt>\ar@{~}[u]!<-5pt,0pt>
                     \ar@{-}[u]!<5pt,0pt>\ar@{-}[u]!<20pt,-1pt> 
}}

\newcommand{\oconeedge}[2]{
\objectmargin={0pt}
\xymatrix{
 {}\save[]+<-16pt,8pt>*{\overbrace{\hspace{30pt}}^{#1}} +<6pt,-10pt>*{...}  
 +<28pt,6pt>*\txt{\tiny{1} \tiny{2} $\cdots$ \tiny{\it #2}} \restore \\
  *[o]-<5pt>{\bullet}\ar@{~}[u]!<-30pt,0pt>\ar@{~}[u]!<-20pt,0pt>\ar@{~}[u]!<-5pt,0pt>     
                    \ar@{-}[u]!<5pt,0pt>\ar@{-}[u]!<12pt,0pt>\ar@{-}[u]!<30pt,0pt> 
}}

\newcommand{\bracedunshuffle}[2]{
\objectmargin={0pt}
\xymatrix@R=18pt@C=16pt{
 {}\save[]+<0pt,6pt>*{\bracedcorolla{#1}}                               \restore & 
 {}\save[]+<4pt,6pt>*{\overbrace{\hspace{20pt}}^{#2}} +<0pt,-8pt>*{...} \restore          \\
 & *[o]-<5pt>{\bullet}\ar@{~}[ul]+0\ar@{~}[u]!<-6pt,0pt>\ar@{~}[u]!<12pt,0pt>\ar@{~}[d] \\
 &
}}

\newcommand{\ocunshufled}[4]{
\objectmargin={0pt}
\xymatrix@R=14pt{
 {}\save[]+<0pt,6pt>*{\overbrace{\hspace{25pt}}^{#1}} +<4pt,-8pt>*{...} \restore     &  
 1 & 2 {\cdots} & #3                                                                       \\
 & *[o]-<5pt>{\bullet}\ar@{~}[u]!<-5pt,0pt>\ar@{~}[u]!<-12pt,0pt>\ar@{~}[u]!<12pt,0pt> 
  \ar@{-}[u]\ar@{-}[ur]\ar@{-}[urr] &   
 {}\save[]+<0pt,6pt>*{\overbrace{\hspace{25pt}}^{#2}} +<4pt,-8pt>*{...} \restore     &  #4   \\
 & & *[o]-<5pt>{\bullet}\ar@{~}[ul]\ar@{~}[u]!<-12pt,0pt>\ar@{~}[u]!<-5pt,0pt>
                                 \ar@{~}[u]!<12pt,0pt>\ar@{-}+0 &                          \\ 
 & & &  
}}

\newcommand{\straightunshufled}[2]{
\objectmargin={0pt}
\xymatrix@R=18pt{
 {}\save[]+<0pt,8pt>*{\overbrace{\hspace{28pt}}^{#1}} +<4pt,-10pt>*{...} \restore     &  \\
 *[o]-<5pt>{\bullet}\ar@{-}[u]!<-5pt,0pt>\ar@{-}[u]!<-12pt,0pt>\ar@{-}[u]!<12pt,0pt> &  
 {}\save[]+<0pt,8pt>*{\overbrace{\hspace{28pt}}^{#2}} +<4pt,-10pt>*{...} \restore    &  \\
 & *[o]-<5pt>{\bullet}\ar@{-}[ul]\ar@{-}[u]!<-12pt,0pt>\ar@{-}[u]!<-5pt,0pt>\ar@{-}[u]!<12pt,0pt> 
                        \ar@{-}[d]!<0pt,8pt> \\ & 
}}

\newcommand{\whistleback}{\objectmargin={0pt}
\xymatrix@R=18pt@C=0pt{ 
 &  *+<1pt>\txt{\tiny{\emph 1}}  &  \\ & *-<5pt>{\bullet}\ar@{~}[u]!<0pt,10pt> \ar@{-}[d]!<0pt,5pt> & \\ & & 
 } 
}

\newcommand{\jacobi}{
\raisebox{1em}{\makebox{\objectmargin={0pt}
$\xymatrix@R=5pt@C=5pt{ 
\hskip 1ex \makebox{\tiny\emph{1}} & &\, \makebox{\tiny\emph{2}}\hskip 2ex  &\, \makebox{\tiny\emph{3}} \hskip 2ex   \\ 
                  & \bullet\ar@{.}[ul]\ar@{.}[ur] &                                                              &              \\ 
                  &                               &  \bullet\ar@{.}[ul]\ar@{.}[uur]\ar@{.}[d]                    &              \\
                  &                               &                                                              &  
}$}} $+$ 
\raisebox{1em}{\makebox{\objectmargin={0pt}
$\xymatrix@R=5pt@C=5pt{ 
\hskip 1ex \makebox{\tiny\emph{2}} & &\, \makebox{\tiny\emph{3}}\hskip 2ex  &\, \makebox{\tiny\emph{1}} \hskip 2ex   \\ 
                  & \bullet\ar@{.}[ul]\ar@{.}[ur] &                                                              &              \\ 
                  &                               &  \bullet\ar@{.}[ul]\ar@{.}[uur]\ar@{.}[d]                    &              \\
                  &                               &                                                              &  
}$}} $+$
\raisebox{1em}{\makebox{\objectmargin={0pt}
$\xymatrix@R=5pt@C=5pt{ 
\hskip 1ex \makebox{\tiny\emph{3}} & &\, \makebox{\tiny\emph{1}}\hskip 2ex  &\, \makebox{\tiny\emph{2}} \hskip 2ex   \\ 
                  & \bullet\ar@{.}[ul]\ar@{.}[ur] &                                                              &              \\ 
                  &                               &  \bullet\ar@{.}[ul]\ar@{.}[uur]\ar@{.}[d]                    &              \\
                  &                               &                                                              &  
}$}}
}

\newcommand{\essaycircle}[2]{
\raisebox{1em}{\makebox{
\xymatrix@R=14pt@C=8pt{ 
 *-<20pt>\txt{\tiny{#1}} &              & *+<1pt>\txt{\tiny{#2}}                        \\
                  & *[o]-<3pt>{\circ}\ar@{~}[ul]+0 \ar@{-}[ur]!<3.5pt,0pt> & \\ 
                  & \ar@{-}[u]+0 &  
}}}
}

\newcommand{\essaybullet}[2]{
\raisebox{1em}{\makebox{
\xymatrix@R=14pt@C=8pt{ 
 *\txt{\tiny{#1}}  &              & \save *\txt{\tiny{#2}} \restore           \\
                  & *[o]-<3pt>{\bullet}\ar@{~}[ul]!<-4.5pt,0pt>\ar@{-}[ur]!<3.5pt,0pt>    & \\ 
                  & \ar@{-}[u]+0 &  
}}}
}

\newcommand{\wleftas}[2]{
\raisebox{1em}{\makebox{\objectmargin={0pt}
\xymatrix@R=18pt@C=8pt{ 
 *+<1pt>\txt{\tiny{#1}}          &                         &                        &  \\ 
 *[o]-<3pt>{\bullet}\ar@{~}[u]!<0pt,-5pt> &                & *+<1pt>\txt{\tiny{#2}} &  \\
 & *[o]-<5pt>{\bullet}\ar@{-}[ul]+0\ar@{-}[ur]!<1pt,-5pt>  &                        &  \\ 
 &  \ar@{-}[u]+0 &
}}}
}

\newcommand{\wrightas}[2]{
\raisebox{1em}{\makebox{\objectmargin={0pt}
\xymatrix@R=18pt@C=8pt{ 
                         &                  & *+<1pt>\txt{\tiny{#2}}                   &  \\ 
 *+<1pt>\txt{\tiny{#1}}  &                  & *[o]-<3pt>{\bullet}\ar@{~}[u]!<0pt,-5pt> &  \\
 & *[o]-<5pt>{\bullet}\ar@{-}[ur]+0\ar@{-}[ul]!<1pt,-5pt>  &                           &  \\ 
 &  \ar@{-}[u]+0 &
}}}
}

\newcommand{\lietree}[2]{
\raisebox{1em}{\makebox{\objectmargin={0pt}
\xymatrix@R=14pt@C=8pt{ 
 *+<1pt>\txt{\tiny{#1}} &  & *+<1pt>\txt{\tiny{#2}} \\ 
 & *[o]-<5pt>{\bullet}\ar@{~}[ul]!<-4.5pt,0pt>\ar@{~}[ur]!<3.5pt,0pt> &    \\ 
 & \ar@{~}[u]+0 &  
}}}
}

\newcommand{\actree}[2]{\objectmargin={0pt}
\xymatrix@R=14pt@C=8pt{ 
 *+<1pt>\txt{\tiny{#1}} &  & *+<1pt>\txt{\tiny{#2}} \\ 
 & *[o]-<5pt>{\bullet}\ar@{~}[ul]!<-4.5pt,0pt>\ar@{-}[ur]!<-2pt,-4pt> &    \\ 
 & \ar@{-}[u]+0 &  
}}

\newcommand{\ascomtree}[2]{
\raisebox{1em}{\makebox{\objectmargin={0pt}
\xymatrix@R=14pt@C=8pt{ 
 *+<1pt>\txt{\tiny{#1}} &  & *+<1pt>\txt{\tiny{#2}} \\ 
 & *[o]-<5pt>{\circ}\ar@{~}[ul]!<-4.5pt,0pt>\ar@{~}[ur]!<3.5pt,0pt> &    \\ 
 & \ar@{~}[u]+0 &  
}}}
}

\newcommand{\astree}[2]{
\raisebox{1em}{\makebox{\objectmargin={0pt}
\xymatrix@R=12pt@C=6pt{ 
 *+<4pt>\txt{\tiny{#1}} &  & *+<4pt>\txt{\tiny{#2}} \\ 
 & *[o]-<3pt>{\bullet}\ar@{-}[ul]!<-3pt,0pt>\ar@{-}[ur]!<3.5pt,0pt> &    \\ 
 & \ar@{-}[u]+0 &  
}}}
}

\newcommand{\comowhistle}[2]{
\raisebox{1em}{\makebox{\objectmargin={0pt}
\xymatrix@R=18pt@C=8pt{ 
 *\txt{\tiny{#1}} &  & *\txt{\tiny{#2}} \\ 
 & *[o]-<5pt>{\circ}\ar@{~}[ul]!<0pt,-5pt>\ar@{~}[ur]!<0pt,-5pt> &    \\ 
 & *[o]-<5pt>{\bullet}\ar@{~}[u]!<0pt,5pt>\ar@{-}[d] &  \\
 &              &
}}}
}

\newcommand{\whistleonas}[2]{
\raisebox{1em}{\makebox{\objectmargin={0pt}
\xymatrix@R=18pt@C=8pt{ 
 *+<1pt>\txt{\tiny{#1}}                     &    & *+<1pt>\txt{\tiny{#2}}                    &  \\ 
 *[o]-<3pt>{\bullet}\ar@{~}[u]!<0pt,-5pt>  &    & *[o]-<3pt>{\bullet}\ar@{~}[u]!<0pt,-5pt> &  \\
 & *[o]-<5pt>{\bullet}\ar@{-}[ur]+0\ar@{-}[ul]+0  &                           &  \\ 
 &  \ar@{-}[u]+0 &
}}}
}

\newcommand{\bracedcorolla}[1]{
\objectmargin={0pt}
\xymatrix@R=14pt{
 {}\save[]+<0pt,6pt>*{\overbrace{\hspace{25pt}}^{#1}} +<4pt,-8pt>*{...} \restore       \\
 *[o]-<5pt>{\bullet}\ar@{~}[u]!<-5pt,0pt>\ar@{~}[u]!<-12pt,0pt>\ar@{~}[u]!<12pt,0pt>     
}}
\newcommand{\wigglyunshufled}[2]{
\objectmargin={0pt}
\xymatrix@R=14pt{
 {}\save[]+<0pt,6pt>*{\overbrace{\hspace{25pt}}^{#1}} +<4pt,-8pt>*{...} \restore     &  \\
 *[o]-<5pt>{\bullet}\ar@{~}[u]!<-5pt,0pt>\ar@{~}[u]!<-12pt,0pt>\ar@{~}[u]!<12pt,0pt> &  
 {}\save[]+<0pt,6pt>*{\overbrace{\hspace{25pt}}^{#2}} +<4pt,-8pt>*{...} \restore     &  \\
 & *[o]-<5pt>{\bullet}\ar@{~}[ul]\ar@{~}[u]!<-12pt,0pt>\ar@{~}[u]!<-5pt,0pt>
                                 \ar@{~}[u]!<12pt,0pt>\ar@{-{~}}[d]!<0pt,6pt> \\ 
 &  
}}

\newtheorem{thm}{Theorem}[subsection]
\newtheorem{lem}[thm]{Lemma}
\newtheorem{prop}[thm]{Proposition}
\newtheorem{cor}[thm]{Corollary}
\theoremstyle{definition}
\newtheorem{defn}[thm]{Definition}

\newtheorem{nota}[thm]{Notation}

\theoremstyle{remark}
\newtheorem{rem}[thm]{Remark}

\theoremstyle{remark}
\newtheorem{Convention}[thm]{Convention}

\xyoption{all}

\begin{document}

\title[OCHA and Leibniz Pairs, towards a Koszul duality]{OCHA and Leibniz Pairs, towards a Koszul duality}
\author{Muriel Livernet}
\address{Universit\'e Paris 13, CNRS, UMR 7539 LAGA, 99 avenue
  Jean-Baptiste Cl\'ement, 93430 
  Villetaneuse, France}
\email{livernet@math.univ-paris13.fr}
\author{Eduardo Hoefel}
\address{Universidade Federal do Paran\'a, Departamento de Matem\'atica C.P. 019081, 81531-990 Curitiba, PR - Brazil }
\email{hoefel@ufpr.br}
\thanks{The first author thanks the hospitality of the Chern Institute at Nankai University, Tianjin, during the 
summer school and conference ``Operads and universal algebra'' and especially Bruno Vallette and Vladimir Dotsenko for fruitful discussions. 
The second author is grateful to the Erwin Sch\"oredinger Institute in Vienna for organizing the 
``2010 Higher Structures Conference'' and to Sergei Merkulov for valuable conversations. The authors are grateful to Joan Milles for his explanations on linear-quadratic operads. The collaboration is granted by the program MathAmSud ``OPECSHA'' coordinated by M. Ronco.}
\keywords{}
\subjclass[2000]{}
\date{\today}
\begin{abstract}
In this paper we study the homology of 2 versions of the swiss-cheese operad. We prove that the zeroth homology of these two versions are Koszul operads and relate this to strong homotopy Lebiniz pairs and OCHA, defined by Kajiura and Stasheff in \cite{KS06a}.
\end{abstract}
\maketitle

%
%
%
%

\setcounter{tocdepth}{2}
\tableofcontents

\section{Introduction}

In the 1963's famous paper of Gerstenhaber \cite{Ger63} is built a Gerstenhaber bracket (Lie bracket up to a shift) on the Hochschild complex $C(A,A)$ of an associative algebra $A$, and then a Gerstenhaber structure on the Hochschild cohomology of an associative algebra $HH(A$). Later on, this leads to the Deligne's conjecture, telling us that since the homology of the liltle disk operad $\cal C$ acts on $HH(A)$ then there should be an operad weakly equivalent to the singular chain complex of $\cal C$ that acts on $C(A,A)$. Different proofs of this fact were given in the last decade by McClure and Smith, Tamarkin, Kontsevich and Soibelman, Berger and Fresse and others.

Shifting the Hochschild complex by one, one gets a differential graded Lie algebra structure on $C(A,A)[1]$. Given a Lie algebra (concentrated in degree $0$ with zero differential), the pair $(L,A)$ together with a degree $0$ morphism of dgLie algebras $\phi: L\rightarrow C(A,A)[1]$ is called a Leibniz pair and has been studied by Flato, Gerstenhaber and Voronov in \cite{FGV95}. Indeed $C(A,A)[1]_0=\End(A)$ and $\phi$ has its image in $\Der(A)$. So the latter is equivalent to a pair $(L,A)$ together with a Lie algebra morphism $\phi:L\rightarrow\Der(A)$.
 Following the general idea of deforming structure maps, Kajiura and Stasheff in \cite{KS06a} and \cite{KS06b}, defined the notion of open-closed homotopy algebras (OCHA) inspired by Zwiebach's open closed string field theory. On the way to the definition of OCHA the notion of $L_\infty$-algebras acting on a $A_\infty$-algebra appeared known also as strong homotopy Lebiniz pairs (SHLP).

Let $A$ be an $A_\infty$-algebra. One can still define its deformation complex, which is bigraded
$C^{k,n}(A,A)=\prod_i \hom_\kfield((A^{\otimes k})_i,A_{i-n})$ endowed with a differential of total degree
$-1$, induced by the $A_\infty$-structure. This complex is still endowed with the Gerstenhaber bracket $C^{k,n}\otimes C^{k',n'}\rightarrow C^{k+k'-1,n+n'}$ and one can shift the first degree by one in order to obtain a dg Lie algebra. We denote by  $C^{>0}(A,A)$ the truncated complex where $k>0$. The shifted complex  $C^{>0}(A,A)[1]$ is a Lie subalgebra of the shifted deformation complex.

An SHLP (Strong Homotopy Leibniz Pair) consists of an
$L_\infty$-algebra $L$, an $A_\infty$-algebra $A$ and an
$L_\infty$-morphism $L \rightarrow C^{>0}(A,A)[1]$, called $A_\infty$-algebras over 
$L_\infty$-algebras by Kajiura and Stasheff.
If we consider the full complex $C^*(A,A)$, then an 
$L_\infty$ morphism: $L \to C^*(A,A)[1]$ is an OCHA 
(Open Closed Homotopy Algebra). \\

\medskip

Given an operad $\cal P$, one can define, under some assumptions, its Koszul cooperad $\cal P^{\ac}$. When the operad $\cal P$ is Koszul one gets a quasi-isomorphism
$$\Omega(\cal P^{\ac})\rightarrow \cal P$$ where $\Omega$ is the cobar construction for cooperads. The operad $\Omega(\cal P^{\ac})$ is often denoted $\cal P_{\infty}$ and algebras over this operad are called {\sl strong homotopy $\cal P$-algebras}. In the appendix of \cite{KS06a}, Markl showed that SHLP correspond to algebras over the operad $\Omega(\LP^{\ac})$ where $\LP$ is the operad for Lebiniz pairs. But he did not prove that the operad is Koszul. In \cite{Dolgushev10}, Dolgushev showed that OCHA correspond to algebras over $\Omega(\cal D)$, where $\cal D$ is a  cooperad that will be specified later, 
but without proving that $\cal D$ is Koszul. The aim of our paper is to prove that the operads involved are Koszul, justifying completely that SHLP and OCHA are ``strong homotopy'' algebras over an operad.
It is important to remark that though the case of SHLP is rather classical in the theory of operads (binary quadratic colored operads) the case OCHA is less classical because the operads involved are not quadratic. The technics used are the ones employed by Galvez-Carillo, Tonks and Vallette in the case of BV-algebras in \cite{GTV09}. In order to prove those theorems we use also the relation made by the second author between the swiss-cheese operad and OCHA in \cite{Hoefel09}. This relation is analogous to
the relation between strong homotopy Lie algebras and the little disk operad.

\medskip

The plan of the paper is the following. In section 2, we review some facts on colored operads and Koszul duality that will be needed for the paper. In section 3 we describe the topological operads involved: the Voronov's swiss-cheese operad $\SCvor$ and the Kontsevich's swiss-cheese operads $\SC$. These two different versions are analogous to the difference between the truncated deformation complex of an associative algebra $C^{>0}(A,A)$ and the deformation complex $C(A,A)$. Sections 4 and 5 are devoted to the study of SHLP and the zeroth homology of $\SCvor$. We define Leibniz pairs and its associated colored operad $\LP$ and prove that its Koszul dual operad is $H_0(\SCvor)$. We prove that $\LP$ is a Koszul operad. Consequently SHLP's are algebras over the cobar construction of the Koszul dual cooperad of $\LP$ which is a minimal resolution of $\LP$. We give also an interpretation of  SHLP's in terms of square zero coderivations. Sections 6 and 7 are devoted to the study of OCHA and the zeroth homology of $\SC$. One has that the operad $H_0(\SC)$ is not quadratic but can be expressed with linear-quadratic relations, so that we can apply results of \cite{GTV09}. We prove in theorem
\ref{T:SCKoszul} that $H_0(\SC)$ is a Koszul operad. This means that one has a quasi-isomorphism
$$\Omega(H_0(\SC)^{\ac})\rightarrow H_0(\SC)$$
where the left hand side is not a minimal model since there is a linear differential coming from the one on $H_0(\SC)^{\ac}$. Nevertheless, taking the Koszul dual operad of $H_0(\SC)$, denoted $H_0(\SC)^!$ we prove that there is a quasi-isomorphism
$$\Omega(\cal D)\rightarrow H_0(\SC)^!,$$
where $\cal D$ is a suspension of the cooperad $H_0(\SC)^*$. The left hand side of the quasi-isomorphism is minimal and algebras over it are OCHA as pointed out by Dolgushev in \cite{Dolgushev10} (see section 7.1). The result, however, is not totally satisfying because $H_0(\SC)^!$ is a differential graded operad and one would like to replace it by its homology. We prove in section 7 that there is a sequence of quasi-isomorphisms
$$\Omega(\cal D)\rightarrow H_0(\SC)^!\rightarrow H_*(H_0(\SC)^!),$$
not of operads but of algebras in the category of 2-collections. We use plainly that all these operads are multiplicative operads.
We prove also that the last quasi-isomorphism can not be a quasi-isomorphism in the category of operads. We interpret $ H_*(H_0(\SC)^!)$
as a suspension of the ``top homology'' of the swiss-cheese operad, completing the analogy with the little disks operad $\cal C$. In fact, up to some
suspension, one has that $\Omega(\Com^*)=\Omega(H_0(\cal C))$ is quasi-isomorphic to $\Lie$, which in turn is a suspension of the top homology 
of $\cal C$.

\medskip

\noindent{\bf Notation.} We work on a field $\kfield$ of characteristic $0$.
\begin{itemize}
\item

 The category  $\dgvs$ is the category of lower graded $\kfield$-vector spaces together with a differential of degree $-1$. We will consider the category of vector spaces and the one of graded vector spaces as full subcategories of $\dgvs$. The vector space $\hom_\kfield(V,W)$ denotes the $\kfield$-linear morphisms between two vector spaces  $V$ and $W$. When $V$ and $W$ are objects in $\dgvs$,
 the differential graded vector spaces of maps from $V$ to $W$ is
 $$\Hom_i(V,W)=\prod_n \hom_\kfield(V_n, W_{n+i})$$ together with the differential
$(\partial f)(v)=\d_W(f(v))-(-1)^{|f|}f(d_Vv)$.

 The graded linear dual of $V$ in $\dgvs$ is  $V^*=\Hom(V,\kfield)$, where $\kfield$ is concentrated in degree $0$ with $0$-differential. Consequently $(V^*)_n=(V_{-n})^*$ and $\partial f(x)=-(-1)^n f(d_Vx)$ for any $f\in (V^*)_n$ and 
  $x\in V_{-n+1}$.

\item
Let $V$ be in \dgvs.
The free non-unital commutative algebra generated by  $V$ is denoted by $S(V)$, whereas the free unital commutative algebra generated by  
$V$ is denoted by $S^+(V)$. Similarly for the notation $T(V)$ and $T^+(V)$ for the free non-unital or unital 
associative algebra generated by $V$.
\item

The symmetric group acting on $n$ elements is denoted by $S_n$.  An element $\sigma\in S_n$ will be denoted by its image notation $(\sigma(1)\ldots\sigma(n))$.
An $\mathbb S$-module is a collection of objects $(M(n))_{n\geq 0}$ in $\dgvs$ such that each $M(n)$ is endowed with an action of the symmetric group $S_n$.

\item The Koszul dual cooperad of an operad $\cal P$, when it is defined is denoted by $\cal P^{\ac}$.
Its Koszul dual operad is denoted by $\cal P^!$.
Similarly, the koszul dual algebra of a coalgebra $C$ , when it is defined is denoted by $C^{\ac}$. The notation $B$ stands for the bar construction (of operads or algebras) and $\Omega$ for the cobar construction (of cooperads or coagebras). All these notation will be specified in the text.
\end{itemize}

%
%

\section{Colored operads, 2-colored operads}

This section is devoted to a short summary concerning 2-colored operads. We give here the notation, definitions and theorems needed for the sequel. We refer mainly to  \cite{vanderlaanPhD}  for the general theory of colored operads. One can found also some comments on colored operads in    \cite{markl04}. We will refer to 
 \cite{Fr04} and \cite{LodVal} for the general treatment of  Koszul duality for operads. We explain in this section how this can be adapted to the colored case.

\subsection{Colored operads}\label{S:I-coll}

\begin{defn}[\cite{vanderlaanPhD}]\label{def:I-coll} Let $I$ be a set. The category $\Fin_I$ is the category whose objects $(X,x_0;i:X\rightarrow I)$ are pointed, 
non-empty finite sets together with a map $i$, and whose morphisms are pointed bijections commuting with $i$. 
An {\sl $I$-collection} is a contravariant functor from $\Fin_I$ to $\dgvs$.
An {\sl $I$-colored operad}  
is an $I$-collection $\cal P$ together with natural transformations, called composition maps

$$\circ_x: \cal P(X,x_0;i_X)\otimes \cal P(Y,y_0;i_Y)\rightarrow \cal P(X\cup_x Y, x_0;i_X\cup_X i_Y)$$
for any $x\in X\setminus\{x_0\}$ such that $i_X(x)=i_Y(y_0)$, 
where $i_X\cup_x i_Y: X\cup_x Y\rightarrow I$ is the map induced by $i_X$ and $i_Y$.
These natural transformations are associative, that is, for any $I$-sets $(X,x_0;i_X)$, $(Y,y_0;i_Y)$ and $(Z,z_0;i_Z)$ and 
$\alpha\in \cal P(X,x_0;i_X)$, $\beta\in \cal P(Y,y_0;i_Y)$ and $\gamma\in \cal P(Z,z_0,i_Z)$
one has
\begin{align*}
(\alpha\circ_x \beta)\circ_y \gamma=\alpha\circ_x (\beta\circ_y \gamma),&\ \textrm{\ if\ } i_X(x)=i_Y(y_0) \textrm{\ and\ } i_Y(y)=i_Z(z_0), \\
(\alpha\circ_x \beta)\circ_{x'} \gamma=(\alpha\circ_{x'} \gamma)\circ_x \beta,&\  \textrm{\ if\ } i_X(x)=i_Y(y_0) \textrm{\ and\ } 
i_X(x')=i_Z(z_0).
\end{align*} 
Furthermore, for any $I$-set $\{x,x_0,i\}$ such that $i(x)=i(x_0)$ there exists a map $\kfield\rightarrow \cal P(\{x,x_0,i\})$
which is a right and left identity for the composition maps.
\end{defn}

\subsection{2-colored operads}\label{S:2-colored}

In this article, we  work with 2-colored operads only. Most of the time the 
colours will be denoted by $\cl$ (for closed) and $\op$ (for open). Let us consider the category $\Sk_2$ whose objects are  
$(n,m;x)$ with $n,m \in \mathbb{N}$ and  
$x\in\{\cl,\op\}$ and morphisms
$$\Sk_2((n,m;x);(n',m';x'))=\begin{cases}S_{n}\times S_{m} &\ \textrm{\ if\ } n=n', m=m', x=x' \\
\emptyset & \ \textrm{else}. \end{cases}$$
A {\sl $2$-collection} is a contravariant functor from $\Sk_2$ to $\dgvs$.
There is an equivalence between the category of $2$-collections and the category of $\{\cl,\op\}$-collections as defined in definition~\ref{def:I-coll}.
The equivalence of categories goes as follows. Let $\cal P: \Fin_{\{\cl,\op\}}\rightarrow \dgvs$ be a contravariant functor. To the object
$(n,m;x)\in\Sk_2$ one associates the  object 
$J_{\{n,m;x\}}=\{0,1,\ldots,n+m\}\in \Fin_{\{\cl,\op\}}$, with the structure map  
$$\begin{array}{lccc}
j:&J_{\{n,m;x\}}&\rightarrow& \{\cl,\op\}\\
 &0& \mapsto &x \\ 
 &1\leq k\leq n&\mapsto & \cl \\
 &n+1\leq k\leq n+m&\mapsto &\op.\\
 \end{array}$$
This defines  a functor $J:\Sk_2\rightarrow\Fin_{\{\cl,\op\}}$. Note that
the category $\Sk_2$  is the skeleton of the category $\Fin_{\{\cl,\op\}}$. The functor $\cal P\mapsto\cal P\circ J$ gives one map of the 
equivalence between the two categories.
 
Conversely, let $\cal P$ be a $2$-collection. Let $(X,x;i)$ be an object in $\Fin_{\{\cl,\op\}}$ and let
$n$ be the number of elements in $i^{-1}(\cl)\setminus{x}$ and $m$ be the number of elements in $i^{-1}(\op)\setminus{x}$. The associated $\{\cl,\op\}$-collection is given by
$$\cal P(X,x;i)=\left(\bigoplus_{\Fin_{\{\cl,\op\}}((X,x;i),(J_{\{n,m;x\}},j))} \cal P(n,m;x)\right)_{S_{n}\times S_{m}},$$
where 
the coinvariant are taken under the simultaneous action of $S_{n}\times S_{m}$ on \hfill\break
$\Fin_{\{\cl,\op\}}((X,x;i),(J_{\{n,m;x\}},j))$ and 
$\cal P(n,m;x)$.

A {\sl 2-colored operad} is a 2-collection together with 
composition maps which are equivariant with respect to the action of the symmetric group, and identity maps
$I_{\cl}: \kfield\rightarrow \cal P(1,0;\cl)$ and $I_{\op}:\kfield\rightarrow \cal P(0,1;\op)$ which are unit for the composition maps . 
We will use the following notation for the compositions

\begin{align*}
\circ_i^\cl: \cal P(n,m;x)\otimes \cal P(n',m';\cl)\rightarrow \cal P(n+n'-1,m+m';x) &,\ \textrm{for}\ 1\leq i\leq n \\
\circ_i^\op: \cal P(n,m;x)\otimes \cal P(n',m';\op)\rightarrow \cal P(n+n',m+m'-1;x) &, \ \textrm{for}\ 1\leq i\leq m
\end{align*}

A $2$-colored operad is {\sl reduced} if $\cal P(0,0;x)=0$  and 
$\cal P(1,0;\cl)=\kfield=\cal P(0,1;\op)$. The $2$-collection defined by 
$$\cal I(n,m;x)=\begin{cases}\kfield &\text{ if }  n=1,m=0, x=\cl,\\
\kfield &\text{ if } n=0,m=1,x=\op,\\
0 & \text{ elsewhere, }\end{cases}$$ 
has the trivial 2-colored operad structure. A $2$-colored operad $\cal P$ is {\sl augmented} if there is a morphism of operads
$\cal P\rightarrow \cal I$. We denote by $\overline{\cal P}$  the kernel of the augmentation map.

\begin{defn} In the $2$-colored case, we will consider (colored) pairs of differential graded vector spaces  $V=(V_\cl,V_\op)$. Let $\cal P$ be a $2$-colored operad. An {\sl algebra over $\cal P$} or a {\sl $\cal P$-algebra} is a pair $V=(V_\cl,V_\op)$ in \dgvs, together with evaluation maps
$$\cal P(n,m;x)\otimes_{S_{n}\times S_{m}} (V_\cl^{\otimes n}\otimes V_\op^{\otimes m})\rightarrow V_{x}$$ 
satisfying associativity and unit conditions. Giving a $\cal P$-algebra is the same as giving a morphism of $2$-colored operads
$\cal P\rightarrow \End_V$ where $\End_V$ is the $2$-colored operad $\End_V(n,m;x)=\Hom(V_\cl^{\otimes n}\otimes V_\op^{\otimes m}, V_{x})$ with the natural action of the symmetric group and the natural compositions maps.
The forgetful functor from $\cal P$-algebras to pairs in $\dgvs$  admits a left adjoint, the free $\cal P$-algebra functor which takes the following form: let $V=(V_\cl,V_\op)$ be a pair in $\dgvs$. For  $x\in \{\cl,\op\}$, one has
\begin{equation}\label{E:free-C-algebra}
\cal P(V)_{x}=\bigoplus_{n,m} \cal P(n,m;x)\otimes_{S_{n}\times S_{m}} 
(V_\cl^{\otimes n}\otimes V_\op^{\otimes m}).
\end{equation}
\end{defn}

The forgetful functor from $2$-colored operads to 2-collections admits a left adjoint, the free 
2-colored operad functor and is
denoted by $\cal F(E)$ for $E$ a 2-collection. It can be described in terms of 2-colored trees (where the edges of the trees are colored)  
and has a natural weight grading $\cal F^{(w)}(E)$ where $w$ is the number of the vertices of the trees.

\subsection{Quadratic 2-colored operad}\label{S:quad} As pointed out in \cite{markl04} and 
\cite{vanderlaanPhD}, in order to treat Koszul duality for $I$-colored operads it is more convenient to view colored operads as $K$-operads 
as was originally defined by Ginzburg and Kapranov in \cite{GinKap94} by setting $K$ to be the semi-simple algebra $K=\oplus_{c\in I} \kfield_c$. 
The vector space $\cal P(n)=\sum_{(X,x;i) | |X|=n+1} \cal P(X,x;i)$ is then a $K-K^{\otimes n}$-bimodule. It is also an $S_n$-module and the collection $(\cal P(n))_n$ forms an $\mathbb S$-module.
 
 Thus the usual theory of Koszul operads applies in the colored context.  Most of the papers dealing with Koszul operads, 
have binary quadratic operads as  inputs. In Fresse's paper \cite{Fr04}, general quadratic operads are considered, though with no linear operations. 
In the last part of our paper, we have linear operations, and we refer to the book of Loday and Vallette \cite{LodVal} for the treatment of those operads.

\begin{defn}\label{D:quad} A {\sl quadratic $2$-colored operad} is a 2-colored operad of the form  $\cal F(E)/(R)$ where $E$ is a 2-collection, $R$ is a $\mathbb S$-submodule of $\cal F^{(2)}(E)$  and $(R)$ is the ideal generated by $R$. There are analogous notions of $2$-colored cooperads,  of free $2$-colored cooperads denoted by $\cal F^c(E)$, and of $2$-colored cooperads cogenerated by a $2$-collection $V$  with corelation $R$ denoted by $C(V,R)$. Any quadratic $2$-colored operad $\cal P$ 
admits a {\sl Koszul dual cooperad} $\cal P^{\ac}$, given by the relation
$$\text{ for } \cal P=\cal F(E)/(R) , \text{ one defines } \cal P^{\ac}=C(sE,s^2R),$$ where $sE$ denotes the suspension of the vector space $E$, that is, $(sE)_n=E_{n-1}$. \\

We say that the operad is binary quadratic if it is of the above form with $E$ binary, that is, concentrated in arity $2$. Namely, $E=E_\cl\oplus E_\op$ with $E_{x}=E(\cl,\cl;x)\oplus E(\cl,\op;x)\oplus E(\op,\cl;x)\oplus E(\op,\op;x)$.  The action of the symmetric group $S_2$ is internal in $E(\cl,\cl;x)$ and $E(\op,\op;x)$
and permutes the components $E(\cl,\op;x)$ and $E(\op,\cl;x)$. In particular, if $E$ is finite dimensional then $\dim(E(\op,\cl;x))=\dim(E(\cl,\op;x))$. 
\end{defn}

\begin{defn}\label{D:Koszuldual} If $E$ is binary and of finite dimension, the {\sl quadratic Koszul dual operad} of $\cal P=\cal F(E)/(R)$ is
$\cal P^!:=\cal F(E^\vee)/(R^{\bot})$ where $E^\vee=E^*\otimes\sgn_2$ where $\sgn$ denotes the sign representation and $R^{\bot}$ denotes the orthogonal of $R$ under the induced pairing $\cal F^{(2)}(E)\otimes \cal F^{(2)}(E^\vee)$ as defined by Ginzburg and Kapranov (see also \cite{LodVal}[chapter 7]). Let  $\Lambda$ be the $\mathbb S$-module suspension, that is, for any $\mathbb S$-module $V$
$$\Lambda(V)(k)=s^{1-k}V(k)\otimes\sgn_k.$$
Recall that if $\circ$ denotes the plethysm for $\mathbb S$-modules, then
$$\Lambda V\circ \Lambda W=\Lambda(V\circ W).$$
Moreover, if $\cal P$ is an operad , then $A$ is a  $\cal P$-algebra if and only if $sA$ is a $\Lambda\cal P$-algebra.
One has the relation:
$(\cal F(E)/(R))^{\ac}=(\Lambda (\cal F(E)/(R))^!)^*.$
\end{defn}

In the non-binary case, the above definition generalizes. Let us define the {\sl Koszul dual operad} $\cal P^!$ of a finite dimensional quadratic operad $\cal P$ as 
\begin{equation}\label{E:Koszuldual}
\cal P^!:=(\Lambda \cal P^{\ac})^* \text{ or equivalently }  \cal P^{\ac}=(\Lambda \cal P^!)^*=\Lambda^{-1}((\cal P^!)^*)
\end{equation}
Loday and Vallette prove that if one starts with a quadratic operad $\cal P$ then $\cal P^!$ is again quadratic.

\begin{prop}\cite[Proposition 724]{LodVal}  Let $\cal P=\cal F(E)/(R)$ be a quadratic operad. The operad $\cal P^!$ is quadratic and
$$\cal P^!=\cal F(s^{-1}\Lambda^{-1}E^*)/(R^\bot).$$
\end{prop}

Note that if $E$ is binary, then $\Lambda^{-1}E^*=\Lambda^{-1}E^*(2)=sE^*\otimes\sgn_2$ so we recover the usual case. If a binary operation has degree $k$ in $\cal P$ then the corresponding operation has degree $-k$ in $\cal P^!$ and $k+1$ in $\cal P^{\ac}$. Note that if $E$ is linear, that is, concentrated in arity $1$, then $s^{-1}(\Lambda^{-1}E^*)=s^{-1}E^*$. Consequently  if a unary operation has degree $k$ in $\cal P$ then the corresponding operation has degree $-k-1$ in $\cal P^!$
and degree $k+1$ in $\cal P^{\ac}$.

\begin{Convention} For a $2$-colored operad $\cal P$, when dealing with $\cal P$-algebras
we will prefer the notation of section \ref{S:2-colored}, that is, the notation $\cal P(n,m;x)$ and compositions $\circ_i^x$.  When dealing with the Koszul duality of operads we will prefer 
the notation of section \ref{S:I-coll} and \ref{S:quad}. Here is an example of notation: $\cal P(\cl,\cl,\op,\cl;\op)$ and compositions $\circ_i$.
\end{Convention}

\begin{rem} All the definitions above make sense if we replace the category $\dgvs$ by the category of topological spaces, except for the equivalence of category between 
$2$-collections and $\{\cl,\op\}$-collections, where one  replaces
$$\cal P(X,x;i)=\left(\bigoplus_{\Fin_{\{\cl,\op\}}((X,x;i),(J_{\{n,m;x\}},j))}
\cal P(n,m;x)\right)_{S_{n}\times S_{m}}$$
by
$$\cal P(X,x;i)=\Fin_{\{\cl,\op\}}((X,x;i),(J_{\{n,m;x\}},j))\times_{(S_{n}\times S_{m})} \cal P(n,m;x).$$
Given a $2$-colored operad in topological spaces, its singular chain complex with coefficients in $\kfield$, or its cellular chain complex if the 
topological operad is an operad in CW-complexes, is an operad in $\dgvs$.
Moreover its homology with coefficients in $\kfield$ is an operad in graded vector spaces. And its degree $0$ homology with coefficients in 
$\kfield$ is an operad in vector spaces.
\end{rem}

%
%
%
%

\section{The Voronov's swiss-cheese and the Kontsevich's swiss-cheese }

Here we provide two definitions for the swiss-cheese operad. The difference between the two will be explored 
in correspondence with the difference between OCHAs and Leibniz Pairs. The swiss-cheese as originally defined by Voronov 
\cite{Voronov99} will be denoted by $\SCvor$ and called the Voronov's swiss-cheese operad. With a minor modification, we get another operad denoted $\SC$ which 
is homotopy equivalent to the operad of Kontsevich's compactification of the configuration space of points 
in the upper closed half plane, introduced in \cite{Kontse03}. We call it the swiss-cheese operad or the Kontsevich's swiss-cheese operad.

Voronov has shown that algebras over $H(\SCvor)$ are pairs $(G,A)$ where $G$ is a Gerstenhaber algebra 
and $A$ is an associative algebra with the structure of an algebra over $G$ (viewed as a commutative ring). 
In the case of $H(\SC)$, the structure of algebra over $G$ is given by a unit, i.e., 
a central algebra homomorphism $w : G \to A$.
Note that being unital as an algebra over $G$ does not mean that $A$ is unital 
as an algebra over the groud field $\bf k$. 
\subsection{Little disks}

Let $D^2$ denote the standard unit disk in the complex plane $\bb C$. By a configuration of $n$ disks in $D^2$ we mean a map 
\begin{equation*}
 d : \coprod_{1 \leqslant s \leqslant n} D^2_s \to D^2
\end{equation*}
from the disjoint union of $n$ numbered standard disks $D^2_1, \dots, D^2_n$ to $D^2$
such that $d$, when restricted to each disk, is a composition of translations and dilations, and such that the images of the different components are disjoint. 
The image of each such restriction is called a little disk.
The space of all configurations of $n$ disks is denoted $\cal D_2(n)$ and is topologized 
as a subspace of $(\bb R^2 \times \bb R^+)^n$ containing the coordinates of the center and 
radius of each little disk. The symmetric group acts on $\cal D_2(n)$ by renumbering the disks. 
For $n=0$, we define $\cal D_2(0) = \emptyset$. 
The $\mathbb S$-module $\cal D_2 = \{ \cal D_2(n) \}_{n \geqslant 0}$ admits a well known structure 
of operad given by gluing configurations of disks into little disks, see \cite{May72} for the original definition.

\begin{defn} For $m,n \geqslant 0$ such that $m + n >0 $, let us define $\cal{SC}(n,m;\op)$ as the space of those 
configurations $d \in \cal D_2(2n + m)$ such that its image in $D^2$ is invariant under complex conjugation 
and exactly $m$ little disks are left fixed by conjugation.
A little disk that is fixed by conjugation must be centered at the real line, in this case
it is called {\it open}. Otherwise, it is called {\it closed}.  
The little disks in $\cal{SC}(n,m;\op)$ are labelled according to the following rules:
\begin{enumerate}[i)]
\item Open disks have labels in $\{1, \dots, m \}$ and closed disks 
have labels in $\{ 1, \dots, 2n \}$. 
\item Closed disks in the upper half plane have labels in $\{ 1, \dots, n \}$. 
If conjugation interchanges the images of two closed disks, their labels must be congruent modulo $n$.
\end{enumerate}
\end{defn}

There is an action of $S_n \times S_m$ on $\cal{SC}(n,m;\op)$ extending the action of
$S_n \times \{ e \}$ on pairs of closed disks having modulo $n$ congruent labels and the action of
$\{ e \} \times S_m$ on open disks. Figure \ref{swiss_disc} illustrates a point in the space 
$\cal{SC}(n,m;\op)$.

\begin{figure}
  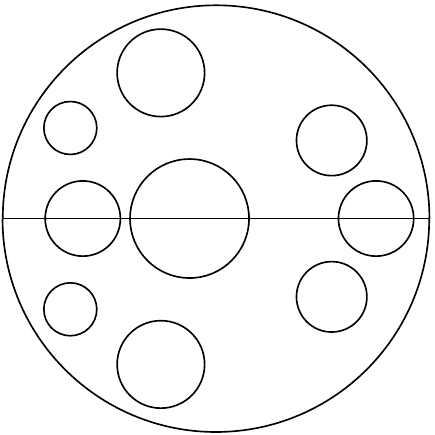
\caption{A configuration in $\cal{SC}(n,m;\op)$}
\label{swiss_disc}
\end{figure}

\begin{defn}[\sl swiss-cheese operad]
 The 2-colored operad $\cal{SC}$ is defined as follows. For $m,n \geqslant 0$ with $m+n > 0$, $\cal{SC}(n,m;\op)$
is the configuration space defined above and $\cal{SC}(0,0;\op) = \emptyset$. For $n \geqslant 0$,  $\cal{SC}(n,0;\cl)$
is defined as $\cal D_2(n)$ and $\cal{SC}(n,m;\cl) = \emptyset$ for $m \geqslant 1$. 
The operad structure in $\cal{SC}$ is given by: 
\begin{align*}
& \circ_i^\cl: \SC (n,m;x)\times \cal{SC}(n',0;\cl)\rightarrow \cal{SC}(n + n'-1,m;x), \ \textrm{for}\ 1\leqslant i\leqslant n \\
& \circ_i^\op: \SC (n,m;x)\times \cal{SC}(n',m';\op)\rightarrow \cal{SC}(n + n',m + m' -1;x), \ \textrm{for}\ 1\leqslant i\leqslant m 
\end{align*}
When $x=\cl$ and $m = 0$, $\circ_i^\cl$ is the usual gluing of little disks in $\cal D_2$. 
If $x=\op$, $\circ_i^\cl$ is defined by gluing each configuration of $\cal{SC}(n',0;\cl)$ in the little disk labelled by $i$ 
and then taking the complex conjugate of the same configuration and gluing the resulting configuration in the little disk 
labelled by $i+n$. Since $\SC (n,m;\cl) = \emptyset$ for $m \geqslant 1$, 
$\circ_i^\op$ is only defined for $x = \op$ and is given by the usual gluing operation of $\cal D_2$. 
\end{defn}

\begin{defn}[\sl Voronov's swiss-cheese operad]
There is a suboperad $\SCvor$ of $\SC$ defined as follows: 
\[ 
\SCvor(n,m;x) := \left\{ 
\begin{array}{ll}
 \SC(n,m;x), &\  \textrm{if}\  x = \cl \ \textrm{or}\  m \geqslant 1 \\
 \emptyset, &\  \textrm{otherwise}.
\end{array}\right.
\]
\end{defn}
The operad $\SCvor$ is the 2-colored operad originally defined by Voronov in \cite{Voronov99}. On the other hand, $\SC$ 
is the 2-colored operad defined by Kontsevich in \cite{Kontse99}, except that, according to his definition, the operad reduces 
to the one point space when $n=0$ and $m=1$. The above definition says that $\SC(0,1;\op)$ is the space of all configurations 
containing only one disk centered at the real line, a contractible space.   

%
%
%
%

\subsection{Configurations of points in the upper closed half-plane}



Here we describe the
Fulton-MacPherson compactification of the configuration space of points in the upper closed 
half-plane introduced by Kontsevich in \cite{Kontse03}. We assume the reader has familiarity
with the Fulton-MacPherson compactification of points in the complex plane
(see \cite{AxelSing94,GetJon94,Merkulov10pr}). 
In the case of points in the upper closed half-plane, a compactification following the guidelines of 
Fulton-MacPherson was introduced by Kontsevich and will be referred to as the 
Kontsevich's compactification in accordance with the terminology used in \cite{Merkulov10pr}). 

Let $p,q$ be non-negative integers satisfying the 
inequality $2p+q \geqslant 2$. We denote by Conf$(p,q)$ the configuration space of marked points
on the upper closed half-plane $H = \{ z \in \mathbb{C} \;|\; {\rm Im}(z) \geqslant 0 \}$ with 
$p$ points in the interior and $q$ points on the boundary (real line): 
\begin{equation*} 
    {\rm Conf}(p,q) = \{ (z_1, \dots, z_p, x_1, \dots, x_q) \in H^{p+q} \;|\;
   z_{i_1} \neq z_{i_2},\, x_{j_1} \neq x_{j_2} \ {\rm Im}(z_i) > 0,\,  {\rm Im}(x_j) = 0.  \}  
\end{equation*}

The above configuration space ${\rm Conf}(p,q)$ is the cartesian product of an open subset of $H^p$ and 
an open subset of $\mathbb{R}^q$ and, consequently, is a $2p + q$ 
dimensional smooth manifold. 
Let $C(p,q)$ be the quotient of ${\rm Conf}(p,q)$ by the action of the 
group of orientation preserving affine transformations that leaves the real line fixed: 
$C(p,q) = {\rm Conf}(p,q)\Big/(z \mapsto az + b)$ where 
$a,b \in \mathbb{R}, \; a > 0$. 
The condition $2p+q \geqslant 2$ ensures that the action 
is free and thus $C(p,q)$ is a $2p + q - 2$ dimensional smooth manifold. In the case of points in the 
complex plane we have: ${\rm Conf}(n) = \{ (z_1, \dots, z_n) \in \C^n \;|\; z_i \neq z_j, \forall i \neq j \}$
and $C(n) = {\rm Conf}(n)\Big/(z \mapsto az + b)$ where 
$a \in \mathbb{R}, \; a > 0 \makebox{ and } b \in \C$. The manifold $C(n)$ is $2n - 3$ dimensional and its 
Fulton-MacPherson compactification is denoted by $\overline{C(n)}$.

Let $\phi$ be the embedding
$\phi : C(p,q)  \longrightarrow C(2p + q)$
defined by 
\begin{equation}\label{embedd}
\phi(z_1 , \dots, z_p, x_1, \dots, x_q) = 
(z_1 ,\bar{z}_1, \dots, z_p, \bar{z}_p, x_1, \dots, x_q) 
\end{equation}
where $\bar z$ denotes complex conjugation. 
The Fulton-MacPherson compactification of $C(p,q)$ is defined as the 
closure in $\overline{C(2p + q)}$ of the image of $\phi$ and 
is denoted by $\overline{C(p,q)}$.

Both compactifications $\overline{C(n)}$ and $\overline{C(p,q)}$ have the structure of manifolds 
with corners whose boundary strata are labelled by trees (see \cite{Hoefel06,Kontse03,Merkulov10pr}).
Such labelling by trees defines a 2-colored operad structure denoted by $\cal H_2$. The set 
of colors is $\{ \op,\cl \}$ and 
\begin{equation*}
\cal H_2(p,q;x) := \left\{ 
\begin{array}{ll}
 \overline{C(p,q)}, &\  \textrm{if}\  x = \op  \         \textrm{and}\   2p + q \geqslant 2, \\
 \overline{C(p)},   &\  \textrm{if}\  x = \cl, \ q = 0 \ \textrm{and}\   p \geqslant 2,      \\
 \emptyset,         &\  \textrm{if}\  x = \cl  \         \textrm{and}\   q \geqslant 1.
\end{array}\right.
\end{equation*}
In addition, we define $\cal H_2(1,0;\cl)$ and $\cal H_2(0,1;\op)$ as the one point space.
The labelling by trees in the boundary strata of $\overline{C(p,q)}$ and $\overline{C(n)}$ gives $\cal H_2$ the structure
of a 2-colored operad. The operads $\cal H_2$ and $\SC$ have the same homology and they are weakly equivalent. 
The proof is similar to
the one given by P. Salvatore in \cite[Prop 4.9]{Salvatore01} for the weak equivalence between the little disk operad $\cal C$ and the Fulton-MacPherson operad.

\medskip

In the sequel, we will use indifferently the operads $\cal H_2$ and $\SC$. The operad $\cal H_2$ will be our geometric model.

\section{The operad $H_0({\SCvor})$ and Leibniz pairs}

We prove in this section that the operad $H_0(\SCvor)$ is a quadratic Koszul operad with Koszul dual the operad of Leibniz pairs.


\subsection{The homology of the Voronov swiss-cheese operad}

The homology of the swiss-cheese operad has been computed by Voronov in \cite{Voronov99},
using the earlier computation of F. Cohen in \cite{CohenF76} of the homology of the configuration space of $p$ points on the complex plane. 
We begin by presenting the operad $H(\SCvor)$ in terms of generators and relations using the language of trees.

For presenting a 2-colored operad using trees, we need colored edges. We will adopt the colors {\it wiggly} 
\raisebox{3pt}{\wiggly} and 
{\it straight} \raisebox{3pt}{\straight} corresponding, respectivelly, to closed and open.
The space $\SCvor(2,0;\cl)$ is homotopy equivalent to $S^1$, the action of $S_2$ being the $\Z/2\Z$-action on $S^1$ via the antipodal map, hence $H(\SCvor(2,0;\cl))$ has one commutative generator in degree 0 
denoted by $f_2=$\raisebox{1.5em}{\comm{\txt{\tiny 1}}{\txt{\tiny 2}}{~}} and another commutative one, in degree 1 denoted by 
\raisebox{1.5em}{\lie{\txt{\tiny 1}}{\txt{\tiny 2}}{~}}. 
The space $\SCvor(0,2;\op)$ is homotopy equivalent to $2$ points and the action of $S_2$ is free. The planar tree
$e_{0,2}=$ \raisebox{1.5em}{\ass{\txt{\tiny 1}}{\txt{\tiny 2}}{-}}\, is a generator of the $S_2$-module
 $H_0(\SCvor(0,2;\op))=H(\SCvor(0,2;\op)).$
The space $\SCvor(1,1;\op)$ is  contractible 
and the only generator of $H(\SCvor(1,1;\op))$ is denoted by $e_{1,1}$=\raisebox{1.5em}{\act{\txt{\tiny 1}}{\txt{\tiny 1}}{-}}. The above elements generate the 
operad $H(\SCvor)$. Using the result of F. Cohen about algebras over $H(\cal D_n)$, A. Voronov has proven that algebras over 
$H(\SCvor)$ are equivalent to a pair of vector spaces $(G,A)$ with the structure described in the following theorem.

\begin{thm}[A. Voronov \cite{Voronov99}]
 An algebra over $H(\SCvor)$ is a pair $(G,A)$ where $G$ is a Gerstenhaber algebra and $A$ is an associative algebra over the commutative ring $G$. 
\end{thm}

We observe that any Gerstenhaber algebra is a commutative ring in particular. In the above theorem, the notion 
of algebra over $G$ corresponds to a map $\rho : G \otimes A \to A$ satisfying the properties stated in the next 
corollary

 \newcommand{\fst}[1]{\txt{\tiny {\emph #1}}}   
 \newcommand{\snd}[2]{{}\save[] +<0pt,7pt>*{\lie{\txt{\tiny \emph #1}}{\txt{\tiny \emph #2}}{}} \restore}   
 \newcommand{\trd}[2]{{}\save[] +<0pt,7pt>*{\act{\txt{\tiny\emph #1}}{\txt{\tiny\emph #2}}{}} \restore}   
 \newcommand{\fth}[2]{{}\save[] +<0pt,7.5pt>*{\ass{\txt{\tiny \emph #1}}{\txt{\tiny \emph #2}}{}} \restore}   
 \newcommand{\fif}[2]{{}\save[] +<0pt,7.5pt>*{\comm{\txt{\tiny\emph #1}}{\txt{\tiny\emph #2}}{}} \restore}   
 
%
%

\begin{cor}[\sl $H_0(\SCvor)$-algebras] \label{L:scvor_alg}  An $H_0(\SCvor)$-algebra consists in the following data: 
a commutative algebra $C$ and an associative algebra $A$ which is a $C$-module, i.e. there is a map $\rho:C\otimes A\rightarrow A$
satisfying 
\begin{equation}\label{eq:C-A-module}
\rho(cc',a)=\rho(c,\rho(c',a))=\rho(c',\rho(c,a)).
\end{equation}

In addition, The structure map $\rho$ satisfies the condition
\begin{equation}\label{eq:C-A}
\rho(c,aa')=\rho(c,a)a'=a\rho(c,a').
\end{equation}

\end{cor}

Using the trees described above, relations (\ref{eq:C-A-module}) and (\ref{eq:C-A}) correspond respectively to:

\[
        \hspace*{\stretch{1}}
        \act{\fif{\it1}{\it2}}{\fst{\it1}}{-} \hskip 1em \raisebox{-8pt}{\emph =} \hskip 1em \act{\fst{\it1}}{\trd{\it2}{\it1}}{-}
        \hskip 1em \raisebox{-8pt}{\emph =} \hskip 1em \act{\fst{\it2}}{\trd{\it1}{\it1}}{-}
        \hskip 2em \raisebox{-8pt}{{\rm \text{and}}} \hskip 2em
        \act{\fst{\it1}}{\fth{\it1}{\it2}}{-}      \hskip 1.5em \raisebox{-8pt}{\emph =} \hskip 2em 
        \ass{\trd{\it1}{\it1}}{\fst{\it2}}{-}      \hskip 1em   \raisebox{-8pt}{\emph =} \hskip 1em 
        \ass{\fst{\it1}}{\trd{\it1}{\it2}}{-}        \hskip 1em   \raisebox{-8pt}{{\rm .}}
        \hspace*{\stretch{2}}
\]

The next  corollary is a consequence of corollary \ref{L:scvor_alg} and formula (\ref{E:free-C-algebra}).

\begin{cor}[\sl Free $H_0(\SCvor)$-algebras]\label{C:freescvor} Let  $V=(V_\cl,V_\op)$ be a pair of graded vector space. 
The free  $H_0(\SCvor)$-algebra generated by the vector space $V$  has the following form:
\begin{align}
H_0(\SCvor)(V)_\cl=&S(V_\cl), \\
H_0(\SCvor)(V)_\op=&S^+(V_\cl)\otimes T(V_\op),
\end{align}
where $S(V_\cl)$ is the free commutative algebra and $H_0(\SCvor)_\op$ is the associative algebra obtained as the tensor product of the two associative algebras $S^+(V_\cl)$ and $T(V_\op)$. The module structure is given by
$$\begin{array}{ccc}
 S(V_\cl)\otimes (S^+(V_\cl)\otimes T(V_\op)) &\rightarrow&  S^+(V_\cl)\otimes T(V_\op) \\
x\otimes (u\otimes v) & \mapsto & xu\otimes v
\end{array}$$
\end{cor}

The next corollary is the transcription  of corollary \ref{L:scvor_alg} in terms of quadratic operads defined in section \ref{S:quad}.

\begin{cor}[\sl Operad description]\label{C:scvor_operad}The operad $H_0(\SCvor)$ has a binary quadratic presentation $\mathcal F(E_v)/(R_v)$ where
$$E_v=\underbrace{kf_2}_{=E_v(\cl,\cl;\cl)}\oplus
\underbrace{k[S_2]e_{0,2}}_{=E_v(\op,\op;\op)} \oplus \underbrace{k[S_2] e_{1,1}.}_{=E_v(\cl,\op;\op)\oplus E_v(\op,\cl;\op)}$$
The action of the symmetric group on $f_2$ is the trivial action.  The representation $k[S_2]$ is the regular representation. The element $e_{1,1}$ forms a basis of $E_v(\cl,\op;\op)$ and $e_{1,1}\cdot (21)$ a basis of $E_v(\op,\cl;\op)$.  The vector space $R_v$ is the submodule of the $S_3$-module$\mathcal F^{(2)}(E_v)$ generated by the relations:
\begin{itemize}
\item associativity of $f_2$: $f_2\circ_2 f_2-f_2\circ_1 f_2$,
\item associativity of $e_{0,2}$: $e_{0,2}\circ_2 e_{0,2}-e_{0,2}\circ_1 e_{0,2}$,
\item  $e_{1,1}$ is an action:

 $e_{1,1}\circ_1^\op e_{1,1}-e_{1,1}\circ_1^\cl f_2,$

$e_{1,1}\circ_1^\op e_{0,2}-e_{0,2}\circ_1^\op e_{1,1}$ and $e_{1,1}\circ_1^\op e_{0,2}-(e_{0,2}\circ_2^o e_{1,1})\cdot (213)$.
\end{itemize}
\end{cor}

%
%
%
%

\subsection{Leibniz pairs}

\begin{defn}\label{D:LP} A {\sl Leibniz pair} consists in the following data: a Lie algebra $L$, an associative algebra $A$ and a Lie morphism 
$L\rightarrow \Der(A)$.
\end{defn}

From the definition of a Leibniz pair one can build the colored operad $\LP$ which is binary quadratic of the form 
$\cal F(E_{lp})/(R_{lp})$ so that  $\LP$-algebras are Leibniz pairs. The collection $E_{lp}$ has the following description:
 $$E_{lp}=\underbrace{k[\sgn]\mathfrak l_2}_{=E_{lp}(\cl,\cl;\cl)}\oplus \underbrace{k[S_2]\mathfrak n_{0,2}}_{=E_{lp}(\op,\op;\op)}\oplus
\underbrace{k[S_2]\mathfrak n_{1,1}}_{E_{lp}(\cl,\op;\op)\oplus E_{lp}(\op,\cl;\op)},$$
 where $\sgn$ denotes the signature representation.
 
The $S_3$-submodule $R_{lp}\in \cal F^{(2)}(E_{lp})$ is generated by the following relations:
\begin{itemize}
\item[(J)] The Jacobi relation $(\mathfrak l_2\circ^\cl_1 \mathfrak l_2)\cdot (\ide+(231)+(312))$,
\item[(A)] The associativity relation $\mathfrak n_{0,2}\circ^\op_1 \mathfrak n_{0,2}-\mathfrak n_{0,2}\circ^\op_2 \mathfrak n_{0,2}$,
\item[(D)] The derivation relation $\mathfrak n_{1,1}\circ^\op_1 \mathfrak n_{0,2}-(\mathfrak n_{0,2}\circ^\op_2 \mathfrak n_{1,1})\cdot (213)-\mathfrak n_{0,2}\circ^\op_1 \mathfrak n_{1,1}$,
\item[(M)] The Lie algebra morphism relation $\mathfrak n_{1,1}\circ^\cl_1 \mathfrak l_2-(\mathfrak n_{1,1}\circ^\op_1 \mathfrak n_{1,1}) \cdot (\ide-(213))$.
\end{itemize}

\begin{lem}The Koszul dual operad of $\LP$ is the operad $H_0(\SCvor)$.
\end{lem}

\begin{proof} The operad $\LP$ is binary quadratic and we can use definition \ref{D:Koszuldual}.
The Koszul dual operad of $\LP$ is
$\LP^!=\cal F(E_{lp}^\vee)/(R_{lp}^{\bot})$. From
$$E_{lp}=k[\sgn]\mathfrak l_2 \oplus k[S_2]\mathfrak n_{0,2}\oplus
k[S_2]\mathfrak n_{1,1},$$
one gets
$$E_{lp}^\vee=kf_2 \oplus k[S_2]e_{0,2}\oplus
k[S_2]e_{1,1}=E_v.$$
The pairing between $E_v$ and $E_{lp}$ induces a pairing 
between $\cal F^{(2)}(E_{lp})$ and 
$\cal F^{(2)}(E_v)$ (see \cite[chapter7]{LodVal}).  One gets that $R_{lp}^{\bot}(\cl,\cl,\cl;\cl)$ is the orthogonal of the Jacobi relation for $\mathfrak l_2$, that is, the associativity relation for $f_2$. Similarly $R_{lp}^{\bot}(\op,\op,\op;\op)$ is the orthogonal of the associativity relation for $\mathfrak n_{0,2}$, that is,
the associativity relation for $e_{0,2}$.

The space $\cal F(E_{lp})(\cl,\cl,\op;\op)$ has dimension 3 and $R_{lp}(\cl,\cl,\op;\op)$ has dimension 1 which is relation (M).
As a consequence the dimension of $R_{lp}^{\bot}(\cl,\cl,\op;\op)$ is 2 and corresponds to  relation~(\ref{eq:C-A-module}).

The space $\cal F(E_{lp})(\cl,\op,\op;\op)$ has dimension 6 and $R_{lp}(\cl,\op,\op;\op)$ has dimension 2 which is  relation (D).
As a consequence the dimension of $R_{lp}^{\bot}(\cl,\op,\op;\op)$ is 4 and corresponds to relation~(\ref{eq:C-A}).
\end{proof}

\begin{lem}\label{L:freeLP} Let $V=(V_\cl,V_\op)$ be a pair of graded   vector space. The free $\LP$-algebra generated by $V$  has the following form:
\begin{align*}
\LP(V)_\cl=&\Lie(V_\cl), \\
\LP(V)_\op=&T(T^+(V_\cl)\otimes V_\op).
\end{align*}
The Lie algebra structure on $\Lie(V_\cl)$ is the free one, and the associative structure on $\LP(V)_\op$
is the free one. The action by derivation of $\LP(V)_\cl$ on $\LP(V)_\op$ is uniquely determined by the action of $V_\cl$ on $T^+(V_\cl)\otimes V_\op$ and is induced by the concatenation $V_\cl\otimes T^+(V_\cl)\rightarrow T^+(V_\cl)$.
\end{lem}

\begin{proof}  We prove that the structure defined above satisfies the universal property with respect to the Leibniz pair structure. Let $(L,A)$ be a Leibniz pair. Let us denote by $l:L\otimes L\rightarrow L$ the Lie bracket, by $\mu:A\otimes A\rightarrow A$ the associative product and by 
$\alpha:L\otimes A\rightarrow A$ the action of $L$ on $A$. Let $\phi:V_\cl\rightarrow L$ and $\psi:V_\op\rightarrow A$ be two linear maps. One has to build a unique pair
$(\tilde\phi:\Lie(V_\cl)\rightarrow L,\tilde\psi:\LP(V)_\op\rightarrow A)$ such that $\tilde\phi$ is a morphism of Lie algebras 
extending $\phi$, such that  $\tilde\psi$ is a morphism of associative algebras extending $\psi$ and such that the following diagram is commutative:
$$\xymatrix{ \LP(V)_\cl\otimes \LP(V)_\op 
\ar[d]_{\tilde\phi\otimes\tilde\psi}\ar[r]^>>>>>{\mathfrak n_{1,1}} & \LP(V)_\op\ar[d]^{\tilde\psi} \\
L\otimes A \ar[r]^{\alpha} & A}$$

Since $\LP(V)_\cl$ is the free Lie algebra generated by $V_\cl$, $\tilde\phi$ is the unique morphism of Lie algebras extending $\phi$. Since $\LP(V)_\op$ is the free associative algebra generated by $T^+(V_\cl)\otimes V_\op$ the morphism $\tilde\psi$ is uniquely determined by its value on $T^+(V_\cl)\otimes V_\op$.
Besides $T^+(V_\cl)\otimes V_\op$ is the free $\Lie(V_\cl)$-module generated by $V_\op$ for the universal enveloping algebra of $\Lie(V_\cl)$ is  $T(V_\cl)$ and the action on $T^+(V_\cl)\otimes V_\op$ is defined as the free action. Consequently $\tilde\psi$ is uniquely determined by its value on $1\otimes V_\op$ which is $\psi$. The commutativity of the diagram follows.
\end{proof}

%
%
%
%

\subsection{The operad $\LP$ is Koszul}\label{S:LPKoszul}

There are different methods for proving that an operad $\cal P$ is Koszul. One of them is to compute the $\cal P$-homology of a free 
$\cal P$-algebra and prove that it is concentrated in degree 0. Though this method is not always the more efficient, it has the advantage in our case to recover the (co)homology  theory for Leibniz pairs as defined by Flato, Gerstenhaber and Voronov in \cite{FGV95}.

\medskip
Let $\cal P$ be a binary quadratic operad and $A$ be a $\cal P$-algebra.
  The $\cal P$-homology of $A$, denoted by $H_*^{\cal P}(A)$, is the homology of the complex $((\cal P^!)^c(sA),\partial)$, where $(\cal P^!)^c(sA)$ is the free  $\cal P^!$-coalgebra cogenerated by the suspension of $A$, and 
  $\partial$ is the unique coderivation of 
$\cal P^!$-coalgebra extending the $\cal P$-algebra structure of $A$. Thus, given a Leibniz pair
$(L,A)$, the complex computing $H_*^{\LP}(L,A)$ is
\[ C_*(L,A)=({\LP^!}^c(s(L,A)),\partial)=(H_0(\SCvor)^c(s(L,A)),\partial). \] 
In proposition \ref{P:homology_LP} we recognize the complex introduced in \cite{FGV95} and in theorem \ref{T:LPKoszul} we prove that the operad $\LP$ is Koszul.

\medskip
The first step of our study is to describe  the inverse map of the bijection
$$\Coder(H_0(\SCvor)^c(V))\rightarrow \Maps(H_0(\SCvor)^c(V),V),$$
when $V=(V_\cl,V_\op)$ is a pair of vector spaces.
Recall that $H_0(\SCvor)^c(V)_\cl=S^c(V_\cl)$ and $H_0(\SCvor)^c(V)_\op=(S^c)^+(V_\cl)\otimes T^c(V_\op)$
where $S^c$ denotes the free cocommutative functor and $T^c$ the free coassociative functor. Denote by $\bar v_{[n]}$
the element $\overline{v_1\otimes\ldots\otimes v_n}$ in $S^c(V)$.
For any subset
$A=\{a_1,\ldots,a_k\}\subset \{1,\ldots,n\}$ the element  $\overline{v_{a_1}\otimes \ldots\otimes  v_{a_k}}\in S^c(V)$ is denoted by $\bar v_A$. For $w_1,\ldots,w_m\in W$ and $1\leq k\leq l\leq m$, the element $w_k\otimes\ldots\otimes w_l$ of $T^c(W)$
is denoted by $w_{(k;l)}$.

\begin{lem} \label{L:liftcoder}
Let $V=(V_\cl,V_\op)$ be a pair of vector spaces. The inverse map of the bijection
$$\Coder(H_0(\SCvor)^c(V))\rightarrow \Hom(H_0(\SCvor)^c(V),V)$$
has the following form. For $\psi: H_0(\SCvor)^c(V)_\cl\rightarrow V_\cl$ and $\phi:  H_0(\SCvor)^c(V)_\op\rightarrow V_\op$
the $H_0(\SCvor)^c(V)$-coderivations $\tilde\psi:  H_0(\SCvor)^c(V)_\cl\rightarrow H_0(\SCvor)^c(V)_\cl$ and 
$\tilde\phi: H_0(\SCvor)^c(V)_\op\rightarrow H_0(\SCvor)^c(V)_\op$ are described by

\begin{equation}\label{E:derclose}
 \tilde\psi(\bar v_{[p]})=\sum_{A\sqcup B=\{1,\ldots,p\},\atop{ A\not=\emptyset}} \psi(\bar v_A)\bar v_B,
\end{equation}

\begin{multline}\label{E:deropen}
 \tilde\phi(\bar v_{[p]}\otimes w_{(1;q)})=\sum_{A\sqcup B=\{1,\ldots,p\},\atop{A\not=\emptyset}}
\psi(\bar v_A)\bar v_B\otimes w_{(1;q)}+\\
\sum_{A\sqcup B=\{1,\ldots,p\}, \atop{0\leq i<j\leq q}} 
\bar v_A\otimes (w_{(1;i)}\phi(\bar v_B\otimes w_{(i+1;j)})w_{j+1;q}).
\end{multline}

\end{lem}

\begin{proof} The projection of $\tilde\psi$ onto $V_\cl$ is $\psi$ and the projection of
$\tilde\phi$ onto $V_\op$ is $\phi$. Hence it is enough to prove that $(\tilde\psi,\tilde\phi)$ is 
a coderivation of the $2$-colored coalgebra $(H_0(\SCvor)^c(V))$. The map $\tilde\psi$ is a coderivation of the cofree cocommutative coalgebra $S^c(V_\cl)$. Let us prove that $\tilde\phi$ is a coderivation of the coassociatvie coalgebra $(H_0(\SCvor)^c(V))_\op=(S^c)^+(V_\cl)\otimes T^c(V_\op)$. Let
$$\begin{array}{lccc}
\Delta_\op:&(H_0(\SCvor)^c(V))_\op& \rightarrow&
 (H_0(\SCvor)^c(V))_\op\otimes (H_0(\SCvor)^c(V))_\op \\
&\bar v_{[p]}\otimes w_{(1;q)}&\mapsto &\sum\limits_{A\sqcup B=[p], \atop{1\leq i\leq q-1}} (\bar v_A\otimes w_{(1;i)})\otimes
(\bar v_B\otimes w_{(i+1;q)})
\end{array}$$
be the dual of the associative product defined in corollary
\ref{C:freescvor}. The relation $\Delta_\op\tilde\phi=(\tilde\phi\otimes 1+1\otimes\tilde\phi)\Delta_\op$ is a direct computation.
Let 
$$\begin{array}{lccc}
\gamma_{\cl,\op}:& (H_0(\SCvor)^c(V))_\op& \rightarrow&  (H_0(\SCvor)^c(V))_\cl\otimes (H_0(\SCvor)^c(V))_\op\\
&\bar v_{[p]}\otimes w_{(1;q)}&\mapsto&\sum\limits_{A\sqcup B=[p],\atop{A\not=\emptyset}} \bar v_A\otimes
(\bar v_B\otimes w_{(1;q)})
\end{array}$$
be the dual of the module structure defined in corollary \ref{C:freescvor}.
The relation $\gamma_{\cl,\op}\tilde\phi=(\tilde\psi\otimes 1+1\otimes\tilde\phi)\gamma_{\cl,\op}$  is also a direct computation.
\end{proof}


\begin{prop}\label{P:homology_LP} Let $(L,A)$ be a Leibniz pair. Let $C_*(L,A)$ be the complex computing the homology $H_*^{\LP}(L,A)$.
The closed component $C_*(L,A)_\cl$ is the Chevalley-Eilenberg complex $C_*^{CE}(L;\kfield)$.  
The open component $C_*(L,A)_\op$ is the bicomplex
$$C_{p,q}(L,A)_\op=C_p^{CE}(L,C_q^{\Hoch}(A)),$$
where $C_q^{\Hoch}(A)$ is the Hochschild complex computing the Hochschild homology of the non-unital associative algebra $A$ with coefficients in $\kfield$.
\end{prop}

\begin{proof} Denote by $[-,-]$ the Lie bracket of $L$  as well as the Lie action on $A$.
The Leibniz pair structure on $V=(L,A)$ induces a natural map $H_0(\SCvor)^c(sV)\rightarrow sV$ which is given componentwise by
$$\begin{array}{llccl}
&\phi:& S^c(sL)&\rightarrow& sL\\
& & & &\\
&& \overline{sl_1\otimes sl_2\ldots\otimes sl_n}&\mapsto& \begin{cases}  s[l_1,l_2], &\ \textrm{if}\ n=2, \\
0, & \ \textrm{otherwise}, \end{cases} \\
&&&&\\
\text{and}& & & & \\
&&&&\\
&\psi:& (S^c)^+(sL)\otimes T^c(sA)& \rightarrow &sA \\
& & & & \\
&&\overline{sl_1\otimes\ldots\otimes sl_n}\otimes sa_1\otimes\ldots sa_k&\mapsto &
\begin{cases} s[l_1,a_1],& \ \textrm{if}\ n=k=1, \\
s(a_1\cdot a_2), & \ \textrm{if}\ n=0 \text{ and } k=2, \\
0, &  \ \textrm{otherwise}. 
\end{cases}
\end{array}
$$
Since $L$ and $A$ are concentrated in degree $0$, then $sL$ and $sA$ are 
concentrated in degree 1. Hence one can use the isomorphism $S^n(sL)\simeq \Lambda^n(L)$.
Formula (\ref{E:derclose}) gives the differential on the closed component of the complex
$d_\cl: C_n(L,A)_\cl=\Lambda^n(L)\rightarrow \Lambda^{n-1}(L),$ by

$$d_\cl(l_1\wedge\ldots\wedge l_n)=\sum_{1\leq i<j\leq n} (-1)^{i+j-1} [l_i,l_j]\wedge l_1\ldots\wedge \hat l_i\wedge\ldots\wedge \hat l_j\wedge\ldots\wedge l_n$$
which is the complex computing the Chevalley-Eilenberg homology of the Lie algebra $L$ with trivial coefficients.

Formula (\ref{E:deropen}) gives the differential on the open component of the complex
$d_\op: C_n(L,A)_\op=\oplus_{p+q=n}\Lambda^p(L)\otimes T^q(A)\rightarrow C_{n-1}(L,A)_\op$ by

\begin{multline*}
d_\op(l_1\wedge\ldots\wedge l_p\otimes a_1\otimes\ldots\otimes a_q)=
\sum_{1\leq i<j\leq p} (-1)^{i+j-1} [l_i,l_j]\wedge l_1\ldots\wedge \hat l_i\wedge\ldots\wedge \hat l_j\wedge\ldots\wedge l_p\otimes a_{(1;q)} + \\
+ \sum_{i=1}^p\sum_{j=1}^q     (-1)^{p-i+j-1} l_1\ldots\wedge \hat l_i\wedge\ldots\wedge l_p\otimes( a_1\otimes\ldots a_{j-1}\otimes [l_i,a_j]\otimes\ldots\otimes a_q) \\
+\sum_{i=1}^{q-1} (-1)^{p+i} l_{[p]}\otimes a_1\otimes\ldots \otimes a_i\cdot a_{i+1}\otimes\ldots a_q
\end{multline*}

\end{proof}

\begin{thm}\label{T:LPKoszul} The operad $\LP$ is Koszul, so is the operad $H_0(\SCvor)$.
\end{thm}

\begin{proof} Consider the free $\LP$-algebra generated by $V=(V_\cl,V_\op)$ given in lemma \ref{L:freeLP}: $\LP(V)_\cl=\Lie(V_\cl)$ and $\LP(V)_\op=T(T^+(V_\cl)\otimes V_\op)$.
The homology of the complex $C_*(\cal LP(V))_\cl$ is concentrated in degree $0$ with value $V_\cl$, because it is the Chevalley-Eilenberg homology of the free Lie algebra $\Lie(V_\cl)$.
Let us compute the homology of the bicomplex
$$C_{p,q}\left(\LP(V)\right)_\op=C^{CE}_{p}\left(\Lie(V_\cl),C_q^{\Hoch}(\LP(V)_\op)\right).$$
Since $\LP(V)_\op$ is a free associative algebra its Hochschild homology is concentrated in degree $0$ with value $M=T^+(V_\cl)\otimes V_\op$. The induced Lie-module structure is the free Lie module structure.
Consequently $H_*^ {CE}(\Lie(V_\cl);T^+(V_\cl)\otimes V_\op)$ is concentrated in degree $0$ with value $V_\op$. The spectral sequence converging to $H_*(\LP(V))_\op$ collapses at page 2 and 
$H_0(\LP(V))_\op=V_\op, \ H_n(\LP(V))_\op=0, n>0$. Since $H_0(\SCvor)$ is the Koszul dual operad associated to $\LP$ it is also Koszul. \end{proof}

%
%
%
%

\section{Strong Homotopy Leibniz Pairs} \label{shlp}
Because $\LP$ is Koszul and $\LP^! = H_0(\SCvor)$, it follows from the theory that the 
cobar construction of the cooperad $\LP^{\ac}=(\Lambda H_0(\SCvor))^*$ (see definition \ref{D:Koszuldual}), is quasi-isomorphic to $\LP$. 
Namely, one has the quasi-isomorphism $$\Omega[(\Lambda H_0(\SCvor))^*]\xrightarrow{\sim} \LP,$$ where $\Omega(C)$, the cobar construction of a cooperad $C$, is the free operad 
generated by $s^{-1}\overline{C}$ with differential  defined as the unique derivation extending the cooperad product in $C$. We will denote the operad $\Omega[(\Lambda H_0(\SCvor))^*]$ by 
$\LP_\infty$. Algebras over this operad
are called {\sl Strong Homotopy Leibniz Pairs} (or SHLP for short).

In this section we recognize SHLP as pair $(L,A)$ where $L$ is an $L_\infty$-algebra and $A$ is an $A_\infty$-algebra
with an $L_\infty$-morphism $L \to C^{>0}(A,A)[1]$, where $C^{>0}(A,A)$, the truncated deformation complex of $A$ has been defined in the introduction and in \cite{KS06a}.

Consequently SHLP is the minimal model of  Leibniz pairs as suggested by Martin Markl in the appendix of \cite{KS06a}.

%
%
%
%

\subsection{Description of $\LP_\infty$ in terms of trees}\label{S:LPinfinity}   From corollary \ref{C:freescvor} one gets
 that the 2-collection structure of $\overline{H_0(\SCvor)}$ is such that $H_0(\SCvor)(n,0;\cl)$
is the trivial representation $\kfield$ of $S_n$, for $n > 1$ and $H_0(\SCvor)(p,q;\op)$ is the tensor product 
$\kfield \otimes \kfield[S_q]$ of the trivial representation of $S_p$ and the regular representation of $S_q$, 
for $p+q > 1$ and $q>0$.

A {\it Partially Planar Tree} is  an isotopy class of 
trees embedded in the Euclidean 3 dimensional space $\R^3$ such that the straight (or planar) edges are 
constrained to be in a fixed plane, say the $xy$-plane. The planar edges will also be called {\it open}  
while the wiggly (or spatial) edges will be called {\it closed}. 
We will say that a spatial (resp. planar) edge has color closed (resp. open).
In the next figure we show the partially planar corollae that are relevant to our operad.
The action of the symmetric group is given by reordering the labels.

\renewcommand{\fst}{{}\save[] +<-28pt,4pt>*{\txt{\tiny 1}}+<20pt,0pt>*{\txt{\tiny 2}}+<18pt,0pt>*{\cdots}+<18pt,0pt>*{\txt{\tiny\it n}} \restore}
\begin{equation*} \raisebox{20pt}{
 \raisebox{-1.5em}{$\mathfrak l_n    \ =\ $} \Closedcorolla{\fst}{~} \ \raisebox{-20pt}{ , } \hskip 20pt    
 \raisebox{-1.5em}{$\mathfrak n_{p,q}\ =\ $} 
 \Rootedoccorolla{\txt{\tiny 1}}{\txt{\tiny 2}}{\txt{\tiny\it p}}{\txt{\tiny 1}}{\txt{\tiny 2}}{\txt{\tiny\it q}}.
}
\end{equation*}
The vector space generated by the above trees with $n > 1$, $p+q > 1$ and $q > 0$ defines a 2-collection isomorphic
to $\overline{H_0(\SCvor)}$.

Since $\LP_\infty$ is the free operad generated by $s^{-1} (\Lambda\overline{H_0(\SCvor}))^*$, it follows that $\LP_\infty$ is generated by the above corollae $\mathfrak l_n$ and $\mathfrak n_{p,q}$ for all $n \geqslant 2$, $p + q \geqslant 2$ and $q > 0$, of degree $n-2$ and $p+q-2$. 

The symmetric group action on $\LP_\infty$ is given by reordering labels and multiplying by the signature of the permutation. 
The grading on $\LP_\infty$ is defined as follows. For each tree $T \in \LP_\infty$, its degree is $|T| =  \# l - \# i - 2$, 
where $\# l$ denotes the number of leaves and $\# i$ denotes the number internal edges of $T$. 
Notice that  $|T \circ_i^x S| = |T| + |S| $ for any label $i$ and color $x$.

The differential  on $\LP_\infty$ is the unique derivation extending the cooperad structure of $(\Lambda H_0(\SCvor))^*$. 
It is not difficult to check that it coincides with the  vertex expansion operator on partially planar trees:
\[
 d T = \sum_{T = T'/e} \pm T'. 
\]
Taking into account the operadic suspension $\Lambda$, the signs are given explicitly in the following formulae.

\renewcommand{\snd}[1]{{}\save[] +<5.4pt,32pt>*{\overbrace{\hspace{16pt}}^{#1}} \restore }   
\renewcommand{\trd}[1]{{}\save[] +<-8pt,4.4pt>*{\closedcorolla{\overbrace{\hspace{32pt}}^{#1}}{}} \restore}   

\begin{equation}\label{LP:diff1} 
d \ \raisebox{20pt}{\closedcorolla{}{~}} \ = \sum_{[p] = I_1 \sqcup I_2\atop{|I_1|>1,|I_2|>0}}  (-1)^{|\sigma|} 
                                             \hskip 15pt \closedcorolla{\trd{I_1}\snd{I_2}}{~}
\end{equation}

and

\renewcommand{\fth}{{}\save[] +<0pt,11pt>*{\occorolla{I_2}{\txt{\tiny {\it i}+1}}{\txt{\tiny {\it j}}}} \restore}
\renewcommand{\fif}{{}\save[] +<0pt,6.5pt>*{\closedcorolla{\overbrace{\hspace{32pt}}^{I_1}}{}} \restore}

\begin{multline}\label{LP:diff2}
d \hskip -10pt \raisebox{20pt}{
\Rootedoccorolla{\txt{\tiny 1}}{\txt{\tiny 2}}{\txt{\tiny\it p}}{\txt{\tiny 1}}{\txt{\tiny 2}}{\txt{\tiny\it q}}
} = \hskip -10pt 
\sum_{\begin{array}{c} \scriptstyle{[p] = I_1 \sqcup I_2}        \\[-3pt] 
                       \scriptstyle{0 \leqslant i \leqslant q-1} \\[-3pt] 
                       \scriptstyle{i+1 \leqslant j \leqslant q}
      \end{array}}  \hskip -10pt (-1)^{|\sigma| + i + |I_1| + i|I_2|}
\hskip -5pt 
\raisebox{20pt}{
\openedge{I_1}{\txt{\tiny 1}}{\txt{\tiny \it i}}{\fth}{\txt{\tiny {\it j}+1}}{\txt{\tiny \it q}}{-}
}  + \\[-2em] +
\sum_{[p] = I_1 \sqcup I_2\atop{|I_1]>1}} (-1)^{|\sigma|} \hskip -10pt 
\raisebox{20pt}{ 
\Rootedoccorolla{\fif}{\hskip 15pt\overbrace{\hspace{20pt}}^{I_2}}{}{\txt{\tiny 1}}{\txt{\tiny 2}}{\txt{\tiny\it q}}
}
\end{multline}
where $\sigma \in S_p$ is the unshuffle partitioning $[p]$ into $I_1 \sqcup I_2$ with ordered subsets $I_1$ and $I_2$
and $|I_j|$ denotes the number of elements in $I_j$.


From the Koszulity of $\LP$ it follows that the counit map $\chi : \LP_\infty \to \LP$ taking $\mathfrak l_2$,
$\mathfrak n_{1,1}$ and $\mathfrak n_{0,2}$ to the corresponding generators in $\LP$ and the remaining corollae 
to zero is a quasi-isomorphism of operads.

Comparing with definition 10 in \cite{KS06a} one gets

\begin{cor}\label{C:SHLPvsALinfty}
 SHLP are $A_\infty$-algebras over $L_\infty$-algebras, where SHLP are algebras over the minimal resolution of $\LP$.
\end{cor}

%
%
%
%

%
%
%

\subsection{Coderivations on $S^c(L)$ and $(S^c)^+(L) \otimes T^c(A)$}\hfill\break

Recall from the operad theory that when $\cal P$ is a Koszul operad, a structure of strong homotopy $\cal P$-algebra on a dgvs $A$ is equivalent to
a square zero coderivation on the cofree $\cal P^{\ac}$ coalgebra cogenerated by $A$ which is the same as the cofree $\cal P^!$-coalgebra cogenerated by $sA$. (see \cite[section 10.1]{LodVal} for further details). In our case, we are in the context of colored operads, which does not increase the difficulty.
Using the description of coderivations in lemma \ref{L:liftcoder}, one can use the general theory to describe SHLP in terms of square zero 2-colored coderivation. Here we give another description in the spirit of second author's paper \cite{Hoefel06}. Namely we would like an alternative description that does not use $2$-colored coderivation. In other words, we give a description of strong homotopy Leibniz pair structures on a pair $(L,A)$ by studying 
the coderivations with respect to the coassociative 
coalgebra structures of $S^c(L)$ and $(S^c)^+(L) \otimes T^c(A)$. 

It is well known that an $L_\infty$-algebra structure on $L$ is equivalent to a degree $-1$
coderivation $\cal D$ in ${\rm Coder}(S^c(sL))$ such that \[ [\cal D,\cal D]=0, \] where the bracket denotes the 
commutator of coderivations. We will give an analogous description for SHLPs 
using ${\rm Coder}(S^c(sL))$ and ${\rm Coder}((S^c)^+(sL) \otimes T^c(sA))$.

\begin{nota}
We will denote by $L^p$ the subspace of weight $p$ elements in $(S^c)^+(L)$ and by $L^{\wedge p}$ the subspace of 
weight $p$ elements in the exterior coalgebra $(\Lambda^c)^+(L)$. In other words, the weight grading of those 
coalgebras will be denoted by $(S^c)^+(L) = \bigoplus_{p \geqslant 0}L^p$ and 
$(\Lambda^c)^+(L) = \bigoplus_{p \geqslant 0}L^{\wedge p}$.
\end{nota}

We first note that a map $g : L^p \otimes A^{\otimes q} \to A$ with $q \geqslant 1$, can be lifted to a coderivation in
${\rm Coder}((S^c)^+(L) \otimes T^c(A))$ as follows. For any 
$\bar v_{[n]} \otimes w_{(1;m)} \in (S^c)^+(L) \otimes T^c(A)$, we have that $\tilde g(\bar v_{[n]} \otimes w_{(1;m)})$
is zero if $n < p$ or $m < q$, otherwise it is given by: 
\begin{equation*}
\tilde g(\bar v_{[n]}\otimes w_{(1;m)})=
\sum_{\begin{array}{c}
           \scriptstyle A\sqcup B=\{1,\ldots,n\}, \ 0\leq i<j\leq m \\[-1ex]
           \scriptstyle \# B = p, \ j - i = q
        \end{array}} \pm \;
\bar v_A\otimes (w_{(1;i)}g(\bar v_B\otimes w_{(i+1;j)})w_{(j+1;m)}),
\end{equation*}
here the sign $\pm$ is given by the action of permutations on $\bar v_{[n]}\otimes w_{(1;m)}$ and 
each $w_i$ in $w_{(1;m)} = w_1 \otimes \cdots \otimes w_m \in A^{\otimes m}$ is homogeneous of degree $|w_i|$.

We denote by ${\rm Coder}_A((S^c)^+(L) \otimes T^c(A))$ the subvector space of ${\rm Coder}((S^c)^+(L) \otimes T^c(A))$
spanned by those coderivations obtained by lifting maps $g : (S^c)^+(L) \otimes T^c(A) \to A$. Hence, as vector spaces we have:
\begin{equation*}
 {\rm Coder}_A((S^c)^+(L) \otimes T^c(A)) \cong \Hom((S^c)^+(L) \otimes T^c(A), A).
\end{equation*}

Since, $\Hom((S^c)^+(L) \otimes T^c(A), A) = \Hom((S^c)^+(L), \Hom(T^c(A), A))$ and 
$\Hom(T^c(A), A)\cong {\rm Coder}(T^c(A))$, we have the following isomorphism of vector spaces:
\begin{equation}\label{coder:maps}
 {\rm Coder}_A((S^c)^+(L) \otimes T^c(A)) \cong \Hom((S^c)^+(L), {\rm Coder}(T^c(A)))
\end{equation}

Let $l_G$ denote the Lie bracket in ${\rm Coder}(T^c(A))$ given by the commutator of coderivations, i.e., 
the Gerstenhaber bracket. Since $(S^c)^+(L)$ is a cocommutative coalgebra with coproduct $\Delta$, 
the convolution bracket given by:
\begin{equation*} 
 [f,g] = l_G \circ (f \otimes g) \circ \Delta, 
\end{equation*}
defines a graded Lie algebra structure on $\Hom((S^c)^+(L), {\rm Coder}(T^c(A)))$. Then the isomorphism  (\ref{coder:maps}) is an isomorphism of graded Lie algebras. 
 
\subsubsection{Semi-direct Product}
Let us briefly recall the definition of the semi-direct product of Lie algebras. 
Given two graded Lie algebras $\mathfrak g$ and $\mathfrak h$, a representation by derivations
of $\mathfrak g$ on $\mathfrak h$ is a degree zero Lie algebra morphism:
\begin{equation*}
\rho : \mathfrak g \to {\rm Der}(\mathfrak h)
\end{equation*}
where ${\rm Der}(\mathfrak h)$ denotes the Lie algebra of derivations on $\mathfrak h$.

\begin{defn}
 If $\rho : \mathfrak g \to {\rm Der}(\mathfrak h)$ is a Lie algebra representation by 
 derivations of the Lie algebra $\mathfrak g$ on the Lie algebra $\mathfrak h$, then their semi-direct 
 product is the Lie algebra structure defined on $\mathfrak g \oplus \mathfrak h$ by the following bracket:
 \begin{equation*}
 [(X_1,Y_1),(X_2,Y_2)] = 
 ([X_1,X_2],\rho(X_1)Y_2 - (-1)^{|X_2||Y_1|}\rho(X_2)Y_1 + [Y_1,Y_2]).
 \end{equation*}
 The semi-direct product will be denoted by $\mathfrak g \ltimes \mathfrak h$.
\end{defn}

A natural Lie algebra representation of 
${\rm Coder}(S^c(L))$ on $\Hom((S^c)^+(L), {\rm Coder}(T^c(A)))$ is defined by 
\[ 
  \rho(\phi)f = f \circ \phi, 
\] 
for any $\phi \in {\rm Coder}(S^c(L))$ and $f \in \Hom((S^c)^+(L), {\rm Coder}(T^c(A)))$. 
Notice that:
\begin{equation*}
  \rho([\phi_1,\phi_2],f) 
                          = \rho(\phi_2, \rho(\phi_1,f)) - (-1)^{|\phi_1| |\phi_2|}
                              \rho(\phi_1, \rho(\phi_2,f)),      
\end{equation*}

and
$
  \rho(\phi,[f,g])  = l_G(f \otimes g)\Delta \circ \phi 
                    = l_G(f \otimes g)(\phi \otimes 1 + 1 \otimes \phi)\Delta 
                    = [\rho(\phi,f), g] + [f, \rho(\phi,g)], 
$
hence $\rho$ is a representation by derivations of ${\rm Coder}(S^c(L))$ on 
$\Hom((S^c)^+(L), {\rm Coder}(T^c(A)))$. In view of the isomorphism (\ref{coder:maps}), 
there is a Lie algebra action by derivations of ${\rm Coder}(S^c(L))$ on 
${\rm Coder}_A((S^c)^+(L) \otimes T^c(A))$.
\begin{thm}\label{SHLP:coder}
An SHLP structure on a pair $(L,A)$ is equivalent to a degree $-1$ element  
\[ 
\cal D \in {\rm Coder}(S^c(sL)) \ltimes {\rm Coder}_A((S^c)^+(sL) \otimes T^c(sA)) 
\]
such that $[\cal D,\cal D] = 0$. 
\end{thm}

\begin{proof}
Using the description of $\LP_\infty$ in terms of trees, an algebra over $\LP_\infty$ consists of two families 
of maps $\{l_n : L^{\wedge n} \to L\}_{n \geqslant 2}$ and $\{n_{p,q}: L^{\wedge p} \otimes A^{\otimes q} \to A\}_{p+q \geqslant 2}$.
The maps $l_n$ and $n_{p,q}$ correspond respectively to the corollae $\mathfrak l_n$ and $\mathfrak n_{p,q}$. 
So, their degrees must be given by $|l_n| = n - 2$ and $|n_{p,q}| = p+q-2$
and they must verify the identities corresponding to the definition of the differential operator of $\LP_\infty$ given by 
formulas (\ref{LP:diff1}) and (\ref{LP:diff2}), which in terms of maps becomes: 
\begin{equation}\label{lp:diff1}
 \partial l_n (\bar v_{[n]})= \sum_{A\sqcup B=\{1,\ldots,n\}, \ \# A = p} \pm \; l_{n-p+1}(l_p(\bar v_A)\bar v_B)
\end{equation}
and 
\begin{multline}\label{lp:diff2}
 \partial n_{n,m}(\bar v_{[n]} \otimes w_{(1;m)}) 
 = \sum_{A\sqcup B=\{1,\ldots,n\}, \ \# A = p} \pm \; n_{n-p+1,m}(l_p(\bar v_A)\bar v_B \otimes w_{(1;m)}) \; + \\
 + \; \sum_{\begin{array}{c}
           \scriptstyle A\sqcup B=\{1,\ldots,n\}, \ 0\leq i<j\leq m \\[-1ex]
           \scriptstyle \# B = p, \ j - i = q
        \end{array}} \pm \;
   n_{n-p,m-q+1}(\bar v_A\otimes (w_{(1;i)}n_{p,q}(\bar v_B\otimes w_{(i+1;j)})w_{j+1;m})),
\end{multline}
where the summation runs over all ordered partitions $A\sqcup B=\{1,\ldots,n\}$ into two sets containing at least two elements 
and 
the sign $\pm$ is obtained by multiplying the signs in formulas (\ref{LP:diff1}) and (\ref{LP:diff2}) by the one given by the action 
of the permutation on $\bar v_{[n]}$ and on $\bar v_{[n]} \otimes w_{(1;m)}$. 

A coderivation $\cal D \in {\rm Coder}(S^c(sL)) \ltimes {\rm Coder}_A((S^c)^+(sL) \otimes T^c(sA))$ of degree $-1$ is given by: 
\[ 
    \cal D = \sum_{p \geqslant 1} \tilde l_p \; + \; \sum_{p \geqslant 0, \; q \geqslant 1} \tilde n_{p,q}.
\]
Since suspension converts the symmetric algebra into the exterior algebra and the coderivation $\cal D$ has degree $-1$ on the 
suspension $(sL,sA)$, it follows that the components of $\cal D$ can be viewed as maps $l_p : L^{\wedge p} \to L$ and 
$n_{p,q} : L^{\wedge p} \otimes A^{\otimes q} \to A$ of degrees $|l_p| = p - 2$ and $|n_{p,q}| = p+q -2$.
Denoting $\cal D_L = \sum_{p \geqslant 1} \tilde l_p$ and $\cal D_A = \sum_{p \geqslant 0, \; q \geqslant 1} \tilde n_{p,q}$, 
we have that $\cal D = \cal D_L + \cal D_A$. Using the definition of the bracket in the semi-direct product
${\rm Coder}(S^c(sL)) \ltimes {\rm Coder}_A((S^c)^+(sL) \otimes T^c(sA))$, it follows that:
\[ 
   [\cal D,\cal D] = [\cal D_L + \cal D_A,\cal D_L + \cal D_A] = 
   [\cal D_L, \cal D_L]  +  2[\cal D_L, \cal D_A] + [\cal D_A, \cal D_A],
\]
where $[\cal D_L, \cal D_L] \in {\rm Coder}(S^c(sL))$ and 
$[\cal D_L, \cal D_A] + [\cal D_A, \cal D_A] \in {\rm Coder}((S^c)^+(sL) \otimes T^c(sA))$. The equation 
$[\cal D,\cal D] = 0$ is equivalent to $[\cal D_L,\cal D_L] = 0$ and 
$2[\cal D_L,\cal D_A] + [\cal D_A,\cal D_A] = 0$.
Up to a factor of $2$ the first identity gives relation (\ref{lp:diff1}) for 
the maps $l_p$ while the second identity gives relation (\ref{lp:diff2})
for the maps $l_p$ and $n_{p,q}$ with differentials $d_L = l_1$ and $d_A = n_{0,1}$.
\end{proof}

\section{The operad $H_0(\SC)$ is Koszul}

In this section we follow closely the article by Imma Galvez-Carrillo, Andy Tonks and Bruno Vallette
\cite{GTV09} in order to prove that the operad $H_0(\SC)$ is Koszul. Indeed this operad is not quadratic,
and we need first to give a description by generators and relations so that it is quadratic and linear. By projecting
the relations onto the quadratic part we obtain an operad $qH_0(\SC)$ which turns out to be Koszul. By definition \cite[Appendix A.3]{GTV09} 
one has that $H_0(\SC)$ is Koszul.

\subsection{The homology of $\SC$}

Let us recall that an algebra $A$ over a commutative ring $R$ is {\it unital} if there is a central ring homomorphism
$u : R \to A$ that is also $R$-linear. As mentioned before,  being unital as an $R$-algebra, does not imply that $A$ is unital as a $\kfield$-algebra.
The following theorem is an easy consequence of both F. Cohen and A. Voronov's computation.

\begin{thm}\label{Hsc}
An algebra over $H(\SC)$ is a pair $(G,A)$ where $G$ is a Gerstenhaber algebra and $A$ is a unital associative 
algebra over the commutative algebra $G$. 
\end{thm}


For our pourposes, it is useful to present the degree zero homology $H_0(\SC)$ using trees.


\renewcommand{\fst}[1]{\txt{\tiny\emph #1}}   
\renewcommand{\snd}[2]{{}\save[] +<0pt,7.0pt>*{\lie{\txt{\tiny \emph #1}}{\txt{\tiny \emph #2}}{}}  \restore}   
\renewcommand{\trd}[1]{{}\save[] +<0pt,8.2pt>*{\whistle{\txt{\tiny\emph #1}}{}}                     \restore}   
\renewcommand{\fth}[2]{{}\save[] +<0pt,7.0pt>*{\ass{\txt{\tiny \emph #1}}{\txt{\tiny \emph #2}}{}}  \restore}   
\renewcommand{\fif}[2]{{}\save[] +<0pt,8.2pt>*{\comm{\txt{\tiny\emph #1}}{\txt{\tiny\emph #2}}{}}   \restore}   
\newcommand  {\six}[1]{{}\save[] +<0pt,41.5pt> *{\whistle{\txt{\tiny\emph #1}}{}}                     \restore}   

\begin{cor}\label{C:Hsc}
The operad $H_0(\SC)$ can be presented as follows:
\[ 
\raisebox{-4pt}{$\cal F \Big( $}    
  \raisebox{14pt}{\ass{\fst 1}{\fst 2}{-}}, 
 \underbrace{ \raisebox{14pt}{\comm{\fst 1}{\fst 2}{~}}}_{\text{ commutative} }, 
  \raisebox{18pt}{\whistle{\fst 1}{-}} \; 
\raisebox{-4pt}{$\Big) \Big/ R$}
\]
where $R$ is the ideal generated by the following relations: 
\begin{enumerate}[a)]
 \item The generators \raisebox{14pt}{$\ass{\fst 1}{\fst 2}{-}$} and \raisebox{14pt}{$\comm{\fst 1}{\fst 2}{~}$} satisfy associativity; \\[3ex]
  \item  
        \hspace*{\stretch{1}}
        $\whistle{\fif{1}{2}}{-} \hskip 1em \raisebox{-8pt}{\emph =} \hskip 1em \ass{\trd{1}}{\six{2}}{-}$
        \hskip 1em \raisebox{-8pt}{{\rm \it{and}}} \hskip 1em
        $\ass{\trd{1}}{\fst{1}}{-}      \hskip 1em \raisebox{-8pt}{\emph =} \hskip 1em 
         \ass{\fst{1}}{\trd{1}}{-}      \hskip 1em \raisebox{-8pt}{{\rm .}}$
        \hspace*{\stretch{4}}
\end{enumerate}
\end{cor}
It follows from the above theorem that an $H_0(\SC)$-algebra is a pair $(C,A)$ where $A$ is an associative unital algebra over the 
commutative algebra $C$.

%
%

\subsection{$H_0(\SC)$ is a quadratic-linear operad}\hfill\break

Let us recall first the theory explained  in \cite{GTV09} for quadratic-linear operads. A {\sl quadratic-linear} operad is of the form $\mathcal F(E)/(R)$ with $R\subset \mathcal F^{(1)}(E)\oplus  \mathcal F^{(2)}(E)$.
Such an $R$ is called quadratic-linear. We ask also that the presentation satisfies
 
\begin{itemize}
\item[(ql1)] $R\cap E=\{0\}$ and \\
\item[(ql2)] $(R\otimes E+E\otimes R)\cap \mathcal F^{(2)}(E) \subset R\cap \mathcal F^{(2)}(E)$.
\end{itemize}

Let $q$ denote the projection $\mathcal F(E)\epi \mathcal F^{(2)}(E)$ and let $qR$ be the image of $R$ under this projection.
The operad $\mathcal F(E)/(R)$ is {\sl Koszul} if it satisfies (ql1) and (ql2) and if
$\mathcal F(E)/(qR)$ is a Koszul quadratic operad. Its Koszul dual cooperad is 
$(\mathcal F(E)/(qR))^{\ac}$ {\bf together with a differential} that will be explained in the next section.

\medskip

In order to apply the theory, one needs to express $H_0(\SC)$ as a quadratic-linear operad, which is not the presentation given in  corollary \ref{C:Hsc}. Indeed, in the presentation given in corollary \ref{C:Hsc}, the relation ``being a morphism of algebras'' is a quadratic-cubical relation, that is, lies in $\mathcal F^{(2)}(E)\oplus\mathcal F^{(3)}(E)$.
We add a new generator in the description of the operad $H_0(\SC)$ in order to replace the quadratic-cubical relation by quadratic-linear relations. This new generator $e_{1,1}$,
will correspond at the level of algebras to the operation $\rho(c;a):=f(c)a$.
Consequently one introduces also new relations in the operad translating the relations  $f(c)a=af(c)=\rho(c;a)$ and
$\rho(c;f(c'))=f(cc')=f(c)f(c')$ present in the algebra setting.

\begin{prop}
The operad $H_0(\SC)$ has a presentation $\mathcal F(E)/(R)$ where
$$E=\underbrace{kf_2}_{=E(\cl,\cl;\cl)}\oplus
\underbrace{k[S_2]e_{0,2}}_{=E(\op,\op;\op)} \oplus \underbrace{k[S_2] e_{1,1}}_{=E(\cl,\op;\op)\oplus E(\op,\cl;\op)} \oplus \underbrace{k\alpha}_{=E(\cl;\op)}$$
The action of the symmetric group on $f_2$ is the trivial action and $k[S_2]$ denotes the regular representation. The element $e_{1,1}$ forms a basis of $E(\cl,\op;\op)$ and $e_{1,1}\cdot (21)$ a basis of $E(\op,\cl;\op)$. 

The space of relations $R$ is the submodule of $\mathcal F^{(1)}(E)\oplus  \mathcal F^{(2)}(E)$
defined by $R=R_v\oplus R(\alpha)$ where $R_v$, defined in corollary \ref{C:scvor_operad}, describes the relations in the presentation of the operad $H_0(\SCvor)$ and $R(\alpha)$ is the $S_2$-submodule of $\cal F(E)$ generated by the following relations:
\begin{itemize}
\item two quadratic-linear relations:  $e_{0,2}\circ_1\alpha-e_{1,1}$ and $(e_{0,2}\cdot (21))\circ_1\alpha-e_{1,1}$,
\item a new quadratic relation: $e_{1,1}\circ^\op \alpha-\alpha\circ_1f_2$. 
\end{itemize}
Moreover this presentation satisfies (ql1) and (ql2).
\end{prop}

The projection of $R=R_v\oplus R(\alpha)$ onto $\mathcal F^{(2)}(E)$ is 
$qR=R_v\oplus R'(\alpha)$ where $R'(\alpha)$ is the submodule of  $\mathcal F^{(2)}(E)$
generated by the relations
\begin{itemize}
\item $e_{0,2}\circ_1\alpha$ and $e_{0,2}\circ_2\alpha$,
\item $e_{1,1}\circ_2 \alpha-\alpha\circ_1f_2$. 
\end{itemize}
Consequently an algebra over this operad is an algebra $(C,A)$ over the operad $H_0(\SCvor)$ together with a linear map $f:C\rightarrow A$ satisfying $f(c)a=af(c)=0$ for all $c\in C, a\in A$ and
$\rho(c;f(c'))=f(cc')$ for all $c,c'\in C$.
As in \cite{GTV09}, the operad $qH_0(\SC)$ is obtained as the result of a distributive law between the operad $H_0(\SCvor)$ and $P(\alpha)$ where $P(\alpha)$ is a free colored operad generated by a 1-dimensional vector space $V$ with basis $\alpha\in V(\cl;\op)$. 
The distributive law is given by

\begin{equation}\label{E:distributive}
\begin{array}{ccc}
 H_0(\SCvor)\circ P(\alpha) &\rightarrow &P(\alpha)\circ H_0(\SCvor) \\
e_{0,2}\circ_1\alpha, e_{0,2}\circ_2 \alpha &\mapsto& 0 \\
e_{1,1}\circ_2 \alpha & \mapsto & \alpha\circ_1f_2.\\
\end{array}
\end{equation}

We sum up the result in the next proposition.

\begin{prop} The operad $qH_0(\SC)$ is the operad $P(\alpha)\circ H_0(\SCvor)$, with the composition given by the distributive law (\ref{E:distributive}).
\end{prop}

\begin{thm}\label{T:SCKoszul} The operad $H_0(\SC)$ is Koszul.
\end{thm}

\begin{proof} From  \cite[Chapter8]{LodVal}, one has that $qH_0(\SC)=P(\alpha)\circ H_0(\SCvor)$ is Koszul for $H_0(\SCvor)$ and $P(\alpha)$ are Koszul colored operads.
\end{proof}

The  Koszul dual  cooperad of the operad
$qH_0(\SC)=P(\alpha)\circ H_0(\SCvor)$
is
$$(qH_0(\SC))^{\ac}= H_0(\SCvor)^{\ac}\circ P(\alpha)^{\ac},$$ 
where $H_0(\SCvor)^{\ac}=(\Lambda\LP)^{*}$ and where $P(\alpha)^{\ac}$ is the cofree $2$-colored 
cooperad cogenerated by an element $\alpha$ of degree $1$ in $P(\alpha)^{\ac}(\cl;\op)$.

%
%
%
%

\subsection{The Koszul dual of $H_0(\SC)$ and it's Koszul resolution.}

In the proof of theorem \ref{T:SCKoszul} we have considered the operad $qH_0(\SC)=\cal F(E)/qR$, where $qR$ is the image of $R$ by the projection $q:\cal F(E)\epi \cal F^{(2)}(E)$. Furthermore, we have seen that algebras over the operad $qH_0(\SC)$ are $H_0(\SCvor)$-algebras $(C,A)$ endowed with a map $f:C\rightarrow A$ such that $f(c)a=af(c)=0$ and $\rho(c,f(c'))=f(cc')$.

Let $\varphi:qR\rightarrow E$ defined by
 
\begin{equation}\label{E:phi}
 \begin{cases}\varphi(e_{0,2}\circ_1 \alpha)=\varphi((e_{0,2}\cdot (21))\circ_1 \alpha)=e_{1,1},\\
 \varphi(R_v)=0,\\ 
 \varphi(e_{1,1}\circ_2 \alpha-\alpha\circ_1f_2)=0.\end{cases}
\end{equation}

The Koszul dual cooperad of $qH_0(\SC)$ is $qH_0(\SC)^{\ac}=C(sE,s^2qR)$, with the notation
of definition \ref{D:quad}. To $\varphi$ is associated the composite map
$$qH_0(\SC)^{\ac} \epi s^2qR\xrightarrow{s^{-1}\varphi}
sE.$$
There exists a unique
coderivation $\widetilde{d}_\varphi\, :\, qH_0(\SC)^{\ac}\rightarrow \cal F^c(sE)$  which extends this map. Moreover, $\widetilde{d}_\varphi$  induces a square zero coderivation $d_\varphi$ on
the Koszul dual cooperad $qH_0(\SC)^{\ac}$.

The Koszul dual  cooperad of $H_0(\SC)$ is by definition

$$H_0(\SC)^{\ac}=(C(sV,s^2qR),d_\varphi).$$

The Koszulity of the operad $H_0(\SC)$ implies the quasi-isomorphism (see \cite[Theorem 38]{GTV09})

\begin{equation}\label{E:Koszulresol}
 \Omega(H_0(\SC)^{\ac})\xrightarrow{\sim} H_0(\SC).
\end{equation}

Note that $ \Omega(H_0(\SC)^{\ac})$ is not a minimal model of $H_0(\SC)$ since its differential has a linear part coming from the differential on $H_0(\SC)^{\ac}$.

In order to understand the differential in the cooperad $H_0(\SC)^{\ac}$ it is usually more convenient to understand the Koszul dual operad $H_0(\SC)^!$. 
Recall from (\ref{E:Koszuldual}) that
$$qH_0(\SC)^!=(\Lambda(qH_0(\SC)^{\ac}))^*=\cal F(s^{-1}\Lambda^{-1}E^*)/(qR)^\bot,$$
with a derivation $\delta_\varphi$ which is the unique derivation of operads extending the map
$$^t\varphi: s^{-1}\Lambda^{-1}E^*\rightarrow \cal F(s^{-1}\Lambda^{-1}E^*)$$
where $^t\varphi$ is a combination of transpose and signed suspension of $\varphi$. Namely, 
$H_0(\SC)^!$ is a differential graded operad and one has

\begin{prop} \label{h0scdual}
A differential graded algebra over $H_0(\SC)^!$ consists in the following data
\begin{enumerate}
\item  a differential graded Lie algebra $(L,[,],d_L)$, 
\item a differential graded associative algebra $(A,d_A)$ 
\item A degree $0$ action $\rho:L\otimes A \rightarrow A$ satisfying the relations
\begin{align*}
\rho(l,aa')=& \rho(l,a)a'+(-1)^{|l||a|}a\rho(l,a') \\
\rho([l,l'],a)=&\rho(l,\rho(l',a))-(-1)^{|l||l'|} \rho(l',\rho(l,a))\\
\end{align*}
\item a degree $-1$ map $f:L \rightarrow A$ satisfying the relations
\begin{align*}
f([l,l'])=&(-1)^{|l|}\rho(l,f(l'))-(-1)^{|l||l'|+|l'|}\rho(l',f(l))\\
d_A(f(l))=&-f(d_Ll)
\end{align*}
\item the relation
$$d_A\rho(l,a)=\rho(d_Ll,a)+(-1)^{|l|}\rho(l,d_Aa)+f(l)a-(-1)^{|a|(|l|+1)}af(l).$$
\end{enumerate}
\end{prop}

\begin{proof}From $qH_0(\SC)^{\ac}=H_0(\SC)^{\ac}\circ \cal P(\alpha)^{\ac}$, one obtains that a graded algebra over $H_0(\SC)^!$ is a graded Leibniz pair $(L,A,\rho)$ (thus  $(a),(b),(c)$ of the proposition) together with a degree $-1$ map
$f:L\rightarrow A$. The first relation of item $(d)$ comes from the transpose of the distributive law (\ref{E:distributive}). The derivation $\delta_\varphi$ is non-zero only when $\rho$ is involved, as seen in the definition of $\varphi$ in (\ref{E:phi}).
If $\cal P$ is a differential graded operad and $A$ is a differential graded $\cal P$-algebra then 
$$d_{A}(\mu(a_1,\ldots,a_n))=(d_{\cal P}(\mu))(a_1,\ldots,a_n)+\sum\pm \mu(a_1,\ldots,d_Aa_i,\ldots,a_n).$$
As a consequence if $d_{\cal P}(\mu)=0$ then $\mu$ preserves the differential on $A$. This is the reason why $L$ is a differential graded Lie algebra, $A$ is a differential associative algebra and $f$ preserves the differential. The relation $(e)$ comes from the definition of $\delta_\varphi$.
\end{proof}

Equivalently, the next proposition gives a presentation of the operad $H_0(\SC)^!$ in terms of generators and relations.

\renewcommand{\fst}[1]{\txt{\tiny\emph #1}}   
\renewcommand{\snd}[2]{{}\save[] +<00.0pt,07.0pt>*{\act{\txt{\tiny\emph #1}}{\txt{\tiny\emph #2}}{}}  \restore}   
\renewcommand{\trd}[1]{{}\save[] +<00.0pt,08.0pt>*{\whistle{\txt{\tiny\emph #1}}{}}                   \restore}   
\renewcommand{\fth}[2]{{}\save[] +<00.0pt,07.0pt>*{\ass{\txt{\tiny\emph #1}}{\txt{\tiny\emph #2}}{}}  \restore}   
\renewcommand{\fif}[2]{{}\save[] +<00.0pt,07.0pt>*{\lie{\txt{\tiny\emph #1}}{\txt{\tiny\emph #2}}{}}  \restore}   
\renewcommand{\six}[1]{{}\save[] +<-1.0pt,08.2pt>*{\whistle{\txt{\tiny\emph #1}}{}}                   \restore}   

\begin{prop}\label{L:hsc}
The differential graded operad $H_0(\SC)^!$ can be presented as follows:
\[ 
\raisebox{-4pt}{$\cal F \Big($}     
\raisebox{14pt}{\lie{\fst 1}{\fst 2}{~}}, 
  \raisebox{14pt}{\ass{\fst 1}{\fst 2}{-}}, 
  \raisebox{14pt}{\act{\fst 1}{\fst 1}{-}}, 
  \raisebox{18pt}{\whistle{\fst 1}{-}} \; 
\raisebox{-4pt}{$\Big) \Big/ R$}
\]
where $\mathfrak n_{1,0}=$ \raisebox{22pt}{\whistle{\fst 1}{-}}\; has degree $-1$ and all the others generators have degree $0$ and where $\mathfrak l_2=$\raisebox{22pt}{\lie{\fst 1}{\fst 2}{~}} is antisymmetric. The ideal
$R$ is  generated by the following relations: 
\begin{enumerate}[a)]
 \item The generator  \raisebox{14pt}{$\lie{\fst 1}{\fst 2}{~}$} satisfies Jacobi identity; \\[3ex]

\item The generator $\mathfrak n_{0,2}=$\raisebox{14pt}{$\ass{\fst 1}{\fst 2}{-}$} satisfies associativity; \\[3ex]
 
\item The generator $\mathfrak n_{1,1}= $ \raisebox{14pt}{\act{\fst 1}{\fst 1}{-}} is an action, that is, satisfies\\[3ex]

  \hskip 0.5em 
   the Leibniz rule:     \act{\fst1 }{\fth12}{-} \hskip 2em   \raisebox{-4pt}{$=$} \hskip 1.2em
        \ass{\snd11}{\fst2 }{-} \hskip 1em   \raisebox{-4pt}{$+$} \hskip .5em
        \ass{\fst1 }{\snd12}{-} \hskip 3em   \\[1em] 
        
the Lie algebra morphism condition:            \hskip 1.5em 
        \act{\fif12}{\fst1 }{-} \hskip 1em   \raisebox{-4pt}{$=$} \hskip .5em
        \act{\fst1 }{\snd21}{-} \hskip 2em   \raisebox{-4pt}{$-$} \hskip 1em
        \act{\fst2 }{\snd11}{-} \hskip 3em   \\[1em] 

 \item  ``The Eye'' relation (see Figure \ref{fig:the-eye}):         \hskip 1.5em 
        \whistle{\fif{1}{2}}{-} \hskip 1.5em  \raisebox{-8pt}{\emph =} \hskip .75em \act{\fst 1}{\six{2}}{-} 
                                \hskip 1em    \raisebox{-8pt}{$ - $  } \hskip .5em \act{\fst 2}{\six{1}}{-} 
                                \hskip 2em    \\[1em]

\item
The differential $d : H_0(\SC)^! \to H_0(\SC)^!$ is defined by: 
$
    d  \raisebox{1.5em}{\act{\fst 1}{\fst 1}{-}}
       \hskip .5em ${\emph =}$ \hskip 1em
       \raisebox{8pt}{\ass{\trd1 }{\fst1}{-}} \hskip 1.5em  - \hskip 1em  \raisebox{8pt}{\ass{\fst1 }{\trd1}{-}}
$\;  and is zero on all the others generators.
\end{enumerate}
\end{prop}

\newcommand  {\sev}[1]{{}\save[] +<0pt,40.0pt> *{\whistle{\txt{\tiny\emph #1}}{}} \restore}   

\renewcommand{\fst}[1]{\txt{\tiny #1}}   

\begin{figure}[p]
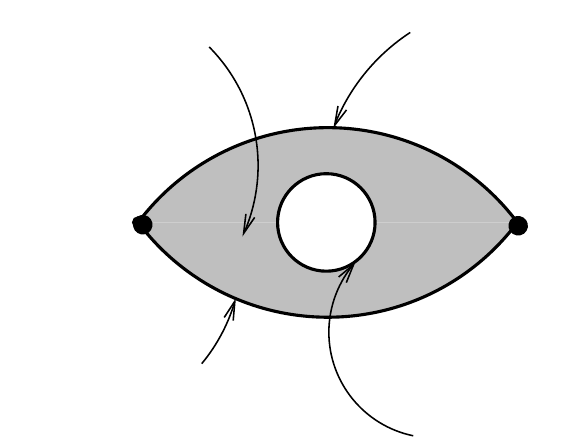 
\vskip 2em
\caption{The manifold $\cal H_2(2,0;\op)$ also called ``The Eye'' and its boudary strata labelled by partially planar trees.}
\label{fig:the-eye}
\end{figure}

The definition of the differential $d : H_0(\SC)^! \to H_0(\SC)^!$ says that the map $f : L \to A$ 
(corresponding to \raisebox{22pt}{\whistle{\fst 1}{-}}) is central 
up to homotopy having the map $\rho: L \otimes A \to A$ (corresponding to
\raisebox{18pt}{\act{\fst 1}{\fst 1}{-}}) as the homotopy operator (see Figure \ref{fig:h-10}).

\begin{figure}[p]
\vskip 2em
\setlength{\unitlength}{3947sp}%
\begingroup\makeatletter\ifx\SetFigFont\undefined%
\gdef\SetFigFont#1#2#3#4#5{%
  \reset@font\fontsize{#1}{#2pt}%
  \fontfamily{#3}\fontseries{#4}\fontshape{#5}%
  \selectfont}%
\fi\endgroup%
\begin{picture}(3500,295)(3079,-1389)
\thinlines
\put(4265,-1348){\circle*{68}}
\put(5956,-1348){\circle*{68}}
\thicklines
\put(4285,-1350){\line( 1, 0){1650}}%
\put(4800,-900){\makebox(0,0)[lb]{\smash{{\SetFigFont{8}{9.6}{\rmdefault}{\mddefault}{\updefault}$\act{1}{1}{-}$}}}}
\put(3700,-1224){\makebox(0,0)[lb]{\smash{{\SetFigFont{8}{9.6}{\rmdefault}{\mddefault}{\updefault}$\ass{\trd 1}{1}{-}$}}}}
\put(6000,-1224){\makebox(0,0)[lb]{\smash{{\SetFigFont{8}{9.6}{\rmdefault}{\mddefault}{\updefault}$\ass{1}{\trd 1}{-}$}}}}
\end{picture}%
\vskip 2em 
\caption{The manifold $\cal H_2(1,1;\op)$ is an interval which parametrizes the up to homotopy centrality 
of $f: L \to A$.}
\label{fig:h-10}
\end{figure}
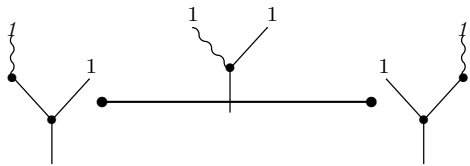

\medskip

Equivalently, the next proposition gives another description of the operad $H_0(\SC)^!$ as a composition of two operads together with a distributive law between them.

\begin{prop}\label{P:distrlaw} The differential graded operad $H_0(\SC)^!$ is the operad $\LP\circ \cal F(\mathfrak n_{1,0})$, where $\mathfrak n_{1,0}$ has degree $1$ and is a generator of $\cal F(\mathfrak n_{1,0})(\op;\cl)$, with the operadic composition given by the distributive law

$$\begin{array}{ccc}
\cal F(\mathfrak n_{1,0})\circ \LP&\rightarrow & \LP\circ \cal F(\mathfrak n_{1,0}) \\
\mathfrak n_{1,0}\circ_1 \mathfrak l_2 & \mapsto & \mathfrak n_{1,1}\circ_2 \mathfrak n_{1,0}-
( \mathfrak n_{1,1}\circ_2 \mathfrak n_{1,0})\cdot (21)\end{array}$$
and the differential given by $d(\mathfrak n_{1,1})= \mathfrak n_{0,2}\circ_1 \mathfrak n_{1,0}-
( \mathfrak n_{1,1}\circ_2 \mathfrak n_{1,0})\cdot (21)$ and $0$ elsewhere.
\end{prop}


\section{Open Closed Homotopy Algebras}\label{OCHA}

In analogy to what we have done in section \ref{shlp}, in this section we will show that an OCHA is an algebra over 
the operad $\Omega[(\Lambda H_0(\SC))^*]$. Here again we will introduce the operad $\mathcal{OC}_\infty$ using 
partially planar trees and will show that this operad coincides with $\Omega(\Lambda H_0(\SC)^*)$. It is then a minimal operad. We prove in theorem \ref{T:minmodOCHA} that it is a resolution of the differential graded operad $H_0(\SC)^!$. Finally we compute its homology, prove that is it a resolution of it in the category of algebras in $2$-collection in theorem \ref{T:homologycalculus}, but not in the category of operads.

\subsection{The OCHA operad}

The OCHA operad $\cal{OC}_\infty$ is generated by the partially planar corollae $\mathfrak l_n$ and $\mathfrak n_{p,q}$ with $n \geqslant 2$ 
and $2p + q \geqslant 2$. 
Those corollae are exactly the same we used for the operad $\LP_\infty$. The only difference is that in the case 
of $\cal{OC}_\infty$ we have generators of the form $\mathfrak n_{p,0}$: 
\renewcommand{\fst}{{}\save[] +<-28pt,4pt>*{\txt{\tiny 1}}+<20pt,0pt>*{\txt{\tiny 2}}+<18pt,0pt>*{\cdots}+<18pt,0pt>*{\txt{\tiny\it p}} \restore}
\[
 \raisebox{-1.5em}{$\mathfrak n_{p,0}    \ =\ $} \Closedcorolla{\fst}{-} \ \raisebox{-20pt}{ . } \hskip 20pt    
\]
In other words, we drop the condition $q > 0$ that we had in the case of strong homotopy Leibniz pairs.
The operadic composition, symmetric group action and differential  of $\cal{OC}_\infty$ are similar to the ones of $\LP_\infty$. The argument used in section \ref{S:LPinfinity} also applies to the following proposition.

\begin{prop} \label{OCHA:iso}
 The operad $\mathcal{OC}_\infty$ coincides with $\Omega(\Lambda H_0(\SC)^*)$.
\end{prop}

%
%
%
%


\begin{thm}\label{T:minmodOCHA}
There is a quasi-isomorphism $\cal{OC}_\infty \to H_0(\SC)^!$. 
\end{thm}
\begin{proof}
Since $H_0(\SC)$ is Koszul, one has the quasi-isomorphism (\ref{E:Koszulresol}):
\begin{equation*}
\Omega(H_0(\SC)^{\ac})\xrightarrow{\sim} H_0(\SC)
\end{equation*}
where $H_0(\SC)^{\ac}$ is a differential graded cooperad related to $H_0(\SC)^!$ by:
$$H_0(\SC)^{\ac}=(\Lambda H_0(\SC)^!)^*.$$
There is an adjunction $(\Omega,B)$ between cooperads and operads, where $B$ denotes the bar construction of an operad and $\Omega$ the cobar construction of a cooperad. The unit of the adjunction is a quasi-isomorphism and the bar construction $B$  preserves quasi-isomorphism (see e.g. \cite{GetJon90}). Furthermore, in the case of finite dimensional operads one has $B\cal P=\Omega(\cal P^*)^*$. Combining these results with equation (\ref{E:Koszulresol}) one gets:
\begin{equation*}
\Lambda H_0(\SC)^{\ac}\xrightarrow{\sim} B\Omega(\Lambda H_0(\SC)^{\ac})\xrightarrow{\sim} B\Lambda H_0(\SC)=\Omega(\Lambda H_0(\SC)^*)^*
\end{equation*}
dualizing the quasi-isomorphism, one gets:
\[\Psi: \cal{OC}_\infty = \Omega(\Lambda H_0(\SC)^*)\xrightarrow{\sim} (\Lambda H_0(\SC))^{\ac})^* = H_0(\SC)^!. \hspace*{\stretch{1}} \qedhere \] 
\end{proof}

\medskip

\renewcommand{\fst}[1]{\txt{\tiny\emph #1}}   
\renewcommand{\snd}[2]{{}\save[] +<00.0pt,07.0pt>*{\act{\txt{\tiny\emph #1}}{\txt{\tiny\emph #2}}{}}  \restore}   
\renewcommand{\trd}[1]{{}\save[] +<00.0pt,08.0pt>*{\whistle{\txt{\tiny\emph #1}}{}}                   \restore}   
\renewcommand{\fth}[2]{{}\save[] +<00.0pt,07.0pt>*{\ass{\txt{\tiny\emph #1}}{\txt{\tiny\emph #2}}{}}  \restore}   
\renewcommand{\fif}[2]{{}\save[] +<00.0pt,07.0pt>*{\lie{\txt{\tiny\emph #1}}{\txt{\tiny\emph #2}}{}}  \restore}   
\renewcommand{\six}[1]{{}\save[] +<-1.0pt,08.2pt>*{\whistle{\txt{\tiny\emph #1}}{}}                   \restore}

Because of the definition of the unit of the adjunction, the quasi-isomorphism
\begin{equation}\label{quism}
 \Psi : \cal{OC}_\infty \to H_0(\SC)^!
\end{equation}
 can be given explicitly in terms of trees in the following way:
on the generators $\mathfrak l_2$, $\mathfrak  n_{1,1}$, $\mathfrak  n_{0,2}$ and $\mathfrak  n_{1,0}$ it acts as the identity map and it vanishes on all the others 
generators of the free operad $\cal{OC}_\infty$.

%
%
%
%

\subsection{The homology of the OCHA operad}\hfill\break

\smallskip

The aim of this section is to prove analogous results for the swiss-cheese operad to the following one for the little disks operad $\cal C$

\begin{itemize}
 \item The zeroth homology of $\cal C$ is Koszul dual to the top homology. More precisely $H_0(\cal C)$ is the operad $\Com$ for commutative algebras.
For any $n$, $H_*(\cal C(n))$ has top homology in degree $n-1$. It is thus suboperad of $H(\cal C)$ denoted by $H_{\rm top}(\cal C)$ 
and coincides with $\Lambda^{-1}\Lie$ where $\Lie$ is the operad for Lie algebras.
\item One has that $\Lie_\infty=\Omega\Lie^{\ac}=\Omega((\Lambda\Com)^*)=\Omega((\Lambda H_0(\cal C))^*)\rightarrow \Lie$
is a quasi-isomorphism.
\end{itemize}

These two items express the same theorem: $\Lie$ and $\Com$ are Koszul dual operads. But the first one gives a geometric interpretation of this fact, whereas the second one gives an interpretation in deformation theory.

We have seen in theorem \ref{T:minmodOCHA} that $H_0(\SC)$ is Koszul and  that there is a quasi-isomorphism of differential graded operads 
$\OC_\infty=\Omega((\Lambda H_0(\SC))^*) \rightarrow H_0(\SC)^!$. It gives a solution to the second item for the swiss-cheese operad, except that $H_0(\SC)^!$ is a differential graded operad and can not be obtained as a suboperad of $H(\SC)$. Nevertheless we introduce the operad $\OC$ that plays the role of the top homology in
$H_*(\SC)$. It does not work as smoothly as for the little disks operad because if we would only take top homology in $H_*(\SC)$, then we 
would not obtain an operad.
We prove in Theorem \ref{T:homologycalculus} that, up to some suspension, the homology of $H_0(\SC)^!$ is $\OC$. 
As a consequence, we have that $H_*(\OC_\infty)=\OC$.
If one had 
a quasi-isomorphism of operads $\OC_\infty\rightarrow \OC$, then one would have the exact analogy with the little disks case.
The non-formality  proved in proposition \ref{P:nonformality}, states that such a quasi-isomorphism does not exist.
Nevertheless, corollary \ref{C:isoalg} shows that there is another structure, algebras in the category of $2$-collections, such that $\OC_\infty\rightarrow \OC$ is a quasi-isomorphism of algebras.

The quasi-isomorphism $\OC_\infty \to \OC$ in the category of $\mathcal{L}_\infty$-modules has been studied
by the second author in \cite{Hoefel09}. It obviously depends on the fact that the homology of $\cal{OC}_\infty$ 
is $\cal{OC}$. Unfortunatelly the proof of the that fact presented in \cite{Hoefel09} contains a mistake. 
The results of the present section constitute a correct proof of the above fact, hence the 
quasi-isomorphism as $\mathcal{L}_\infty$-modules holds as well.

\begin{defn}
 We will denote by $\cal{OC}$ the suboperad of $H(\SC)$ generated by top dimensional homology classes, that is, by the 
 generators of $H_1(\SC(2,0;\cl))$, $H_0(\SC(0,2;\op))$ and $H_0(\SC(1,0;\op))$. 
\end{defn}

Before we proceed, it is necessary to introduce the concept of suspension with respect to a color. 
We will restrict attention to 2-collections, but it is not difficult to extend the definition to a 
more general context. 

\begin{defn}\label{E:ColorS}
Given a 2-collection $\cal P$, the suspension with respect to the color $\cl$ will be denoted by $\Lambda_{\cl} \cal P$ and is defined 
as follows:
\begin{equation*}
  \Lambda_{\cl} \cal P (n,m;x) = \left\{ \begin{array}{ll} 
                                           s^{1-n} P (n,m;x) \otimes ({\rm sgn_n} \otimes \kfield)  & {\rm if}\  x = \cl, \\
                                           s^{-n} P (n,m;x) \otimes  ({\rm sgn_n} \otimes \kfield)  & {\rm if}\  x = \op.
                                         \end{array} \right.
\end{equation*}
where $({\rm sgn_n} \otimes \kfield)$ is the one dimensional representation of $S_n \times S_m$ given by the tensor product
of the signature representation of $S_n$ and the trivial representation of $S_m$.

It is readily seen that a structure of $\cal P$-algebras on the pair $(V_\cl,V_\op)$ is equivalent to a structure of $\Lambda_\cl \cal P$-algebra on the pair $(sV_\cl,V_\op)$.
\end{defn}

The next lemma follows from the description of  $H(\SC)$ given in theorem \ref{Hsc}.

\renewcommand{\fst}[1]{\txt{\tiny\emph #1}}   
\renewcommand{\snd}[2]{{}\save[] +<00.0pt,07.0pt>*{\act{\txt{\tiny\emph #1}}{\txt{\tiny\emph #2}}{}}  \restore}   
\renewcommand{\trd}[1]{{}\save[] +<00.0pt,08.0pt>*{\whistle{\txt{\tiny\emph #1}}{}}                   \restore}   
\renewcommand{\fth}[2]{{}\save[] +<00.0pt,07.0pt>*{\ass{\txt{\tiny\emph #1}}{\txt{\tiny\emph #2}}{}}  \restore}   
\renewcommand{\fif}[2]{{}\save[] +<00.0pt,07.0pt>*{\lie{\txt{\tiny\emph #1}}{\txt{\tiny\emph #2}}{}}  \restore}   
\renewcommand{\six}[1]{{}\save[] +<-1.0pt,08.2pt>*{\whistle{\txt{\tiny\emph #1}}{}}                   \restore}   

\begin{lem}
The operad $\Lambda_\cl\cal{OC}$ has the following presentation
\[ 
\raisebox{-4pt}{$\cal F \Big( $}    
  \raisebox{14pt}{\lie{\fst 1}{\fst 2}{~}}, 
  \raisebox{14pt}{\ass{\fst 1}{\fst 2}{-}}, 
  \raisebox{18pt}{\whistle{\fst 1}{-}} \; 
\raisebox{-4pt}{$\Big) \Big/ R$}
\]
where the generator $\tilde{\mathfrak l}_2=$\raisebox{18pt}{\lie{\fst 1}{\fst 2}{~}} has degree $0$ and is antisymmetric, the generator $\tilde{\mathfrak n}_{0,2}=$   \raisebox{18pt}{\ass{\fst 1}{\fst 2}{-}} has degree $0$
and the generator $\tilde{\mathfrak n}_{1,0}=$  \raisebox{24pt}{\whistle{\fst 1}{-}} has degree $-1$. \\[1em]

The ideal $R$ is generated by the following relations:
\begin{enumerate}[a)]
 \item \hskip 1em
       The generator \raisebox{14pt}{$\ass{\fst 1}{\fst 2}{-}$} satisfies associativity and 
       \raisebox{14pt}{$\lie{\fst 1}{\fst 2}{~}$} satisfies Jacobi identity; \\[3ex]
 \item \hskip 1em
       \raisebox{4pt}{\ass{\trd1 }{\fst1}{-}} \hskip 1.5em  $-$ \hskip 1em  \raisebox{4pt}{\ass{\fst1 }{\trd1}{-}}
       \hskip 1em ${\emph =}$ \hskip .5em 0, \hskip 1em the centrality of \raisebox{20pt}{\whistle{\fst1 }{-}} .
\end{enumerate}

\end{lem}

\begin{thm} \label{T:homologycalculus}
 The homology of the operad $H_0(\SC)^!$ is the operad $\Lambda_\cl\cal{OC}$. \\
 
\end{thm}
\begin{proof} Let $\cal P$ denote the differential graded operad  $H_0(\SC)^!$.
The proof goes in several steps. The first step consists in proving that there is a map of operads
$\phi:\Lambda_\cl\cal{OC}\rightarrow H_*(\cal P)$  inducing a map of algebras in the category of $\{\cl,\op\}$-collections. In the second step, we identify the algebra $\cal P$ with the cobar construction of a coalgebra $C$ in  the category of $\{\cl,\op\}$-collections. The third step consists in proving that $C$ is a Koszul coalgebra, inducing that the algebra map
$\psi:\cal P=\Omega C\rightarrow C^{\ac}$ is a quasi-isomorphism, where $C^{\ac}$ is the Koszul dual algebra associated to $C$, following ideas of Priddy in  \cite{Priddy70}. In the fourth step we identify  $C^{\ac}$ with
 $\Lambda_c\cal{OC}$ and prove that $\psi\phi$ is the identity morphism. Consequently, $\phi$ is an isomorphism of algebras, thus an isomorphism of operads.
 
\medskip

\noindent{\sl First step}. Because $d\mathfrak l_2=d\mathfrak n_{0,2}=d\mathfrak n_{0,1}=0$ in $\cal P$, one has a well defined morphism of operads from $\cal F(\tilde{\mathfrak l}_2, \tilde{\mathfrak n}_{0,2},\tilde{\mathfrak n}_{0,1})\rightarrow H_*(\cal P)$. Because $d\mathfrak n_{1,1}=\mathfrak n_{0,2}\circ_1 \mathfrak n_{0,1}-\mathfrak n_{0,2}\cdot (21)\circ_1 \mathfrak n_{0,1}$, this morphism of operads is well defined on the quotient by the ideal of relations in $\Lambda_\cl\cal{OC}$, yielding an operad morphism $\phi:\Lambda_\cl\cal{OC}\rightarrow H_*(\cal P)$.
It is clear also that the closed part of $\phi$ is an isomorphism. So in the sequel we focus 
on the open part $\phi_\op$ of the morphism. We will prove that given any object $\alpha:(X,\op)\rightarrow \{\cl,\op\}$ in $\Fin_{\{\cl,\op\}}$, the map $\phi_{(X,\op)}$ is an isomorphism. Let us denote the triple $(X,\op,\alpha)$ by $(I,J)$ where $I=\alpha^{-1}(\cl),J=\alpha^{-1}(\op)$. Contravariant functors from the full subcategory of $\Fin_{\{\cl,\op\}}$ generated by these objects will be called {\sl open $2$-collections}.

\medskip

\noindent{\sl Algebras in the category of open $2$-collections}.
The category of open $2$-collections is endowed with a coproduct $(F\oplus G)(I,J)=F(I,J)\oplus G(I,J)$ and a symmetric tensor product
$$(F\otimes G)(I,J)=\bigoplus_{I_1\sqcup I_2=I,\atop{J_1\sqcup J_2=J}} F(I_1,J_1)\otimes G(I_2,J_2)$$
with the unit $U(I,J)=\kfield$ if $I=J=\emptyset$ and is $0$ elsewhere. This tensor product is bilinear with respect to the coproduct.
Consequently  it makes sense to consider associative algebras, commutative algebras, coassociative coalgebras in this category. In the sequel we will use the terminology associative algebras, coassociative coalgebras for associative algebras and coalgebras in the category of open $2$-collections.
Furthermore given an open $2$-collection $V$ one can define an open 
$2$-collection $T(V)$ and $T^+(V)$ as
$T(V)=\oplus_{n\geq 1} V^{\otimes n}$ and $T^+(V)=U\oplus T(V)$. The open $2$-collection $T(V)$ is endowed with the concatenation product so that it becomes the free associative algebra generated by $V$.

Because 
$\tilde{\mathfrak n}_{0,2}$ and $\mathfrak n_{0,2}$ are  associative elements in $\Lambda_\cl\cal{OC}$ and  $\cal P$, those operads are associative algebras in the category of $2$-collections, so is $H_*(\cal P)$. The morphism $\phi$ is then a morphism of associative algebras.

\medskip

\noindent{\sl Second step}. Recall from proposition \ref{P:distrlaw}  that $\cal P=\LP\circ \cal F(\mathfrak n_{1,0})$.  Let $\mu_n\in\LP(0,n;\op)$ denote the $(n-1)$-th composite of $\mathfrak n_{0,2}$ (no matter the way we compose since $\mathfrak n_{0,2}$ is an associative operation). Let $\gamma_k\in\LP(k,1;\op)$ be defined by induction as 
$$\begin{cases}\gamma_0=1_{\op;\op}\\
 \gamma_{k+1}=\gamma_k\circ^\op \mathfrak n_{1,1}.  
  \end{cases}$$ 
Recall from lemma \ref{L:freeLP} that the open part of the free Leibniz pair generated by a pair $V=(V_\cl,V_\op)$ is $\LP(V)_\op=T(T^+(V_\cl)\otimes V_\op)$. Consequently,
any element in $\LP(I,J;\op)$ writes $\mu_n(\gamma_{l_1},\ldots,\gamma_{l_n})\cdot(\sigma,\tau)$ where $n=|J|$, $l_1+\ldots+l_n=|I|$, and $\sigma$ and $\tau$ are bijections between $I$ and $\{1,\ldots,|I|\}$ and between $J$ and $\{1,\ldots,n\}$ respectively. Let 
$$\kappa_{k+1}=\gamma_k\circ^\op \mathfrak n_{0,1}\in \cal P(k+1,0;\op).$$
Then any element in $\cal P(I,J;\op)$ writes 
$$\mu_n(\epsilon_{l_1},\ldots,\epsilon_{l_n})\cdot(\sigma,\tau),$$ 
where $\epsilon_k\in \{\gamma_k,\kappa_k\}$ and $|J|=|\{r|\epsilon_{l_r}=\gamma_{l_r}\}|$ and 
$|I|=l_1+\ldots +l_n$. We recall that $\gamma_k$ and $\mu_n$ are of degree $0$ while $\kappa_k$ is of degree -1.

Let $X_\cl$ be the $2$-collection defined by $X_\cl(I,J)=\kfield$ in degree $0$, if $|I|=1, |J|=0$ and $0$ elsewhere. 
Let
$X_\op$ be the $2$-collection defined by $X_\op(I,J)=\kfield$ of degree 1 if $|I|=0, |J|=1$ and $0$ elsewhere.

With this notation, the algebra $\cal P$ is the free associative algebra generated by the $2$-collection 
$s^{-1}\overline{C}$ with
$$C=(T^+(X_\cl)\oplus T^+(X_\cl)\otimes X_\op).$$  
Recall that the differential on $\cal P$ is a derivation of the operad $\cal P$. Consequently it is a derivation of algebras:

$$d(\mu_n(\epsilon_{l_1},\ldots,\epsilon_{l_n}))=\sum_{i=1}^n (-1)^{r_i} \mu_n(\epsilon_{l_1},\ldots,d\epsilon_{l_i},\ldots,\epsilon_{l_n}),$$
where $r_i$ is the sum of the degrees of $\epsilon_{l_k}$ for $k<i$.
Let us describe $d\gamma_n$ and $d\kappa_n$.

Let $I=\{1,\ldots,n\}$. An element in $T(X_\cl)(I,\emptyset)$ is uniquely determined by a sequence
$i_{[n]}=(i_1,\ldots,i_n)$ of degree 0. One has  $s^{-1}i_{[n]}=\kappa_n\cdot\sigma$ in $\cal P(n,0;\op)$, where $\sigma$ is the permutation $j\mapsto i_j$.
Similarly, an element in $(T^+(X_\cl)\otimes X_\op)(I,\{x\},\op)$ writes $(i_{[n]};x)=(i_1,\ldots i_n; x)$ of degree $1$. One has
$s^{-1}(i_{[n]};x)=\gamma_n\cdot (\sigma,1)$ in  $\cal P(n,1;\op)$.
For a set $A=\{a_1<\ldots<a_k\}\subset\{1,\ldots,n\}$ we denote by $i_A$ the sequence
$(i_{a_1},\ldots, i_{a_k})$ and $(i_A;x)$ the sequence $(i_{a_1},\ldots, i_{a_k};x)$.
One has 

\begin{align}
\label{F1} d(s^{-1}i_{[n]})=&\bigoplus_{A\sqcup B=[n]\atop{A,B\not=\emptyset}} s^{-1} i_A\otimes s^{-1} i_B \\
\label{F2} d(s^{-1}(i_{[n]};x))=&\bigoplus_{A\sqcup B=[n]\atop{A\not=\emptyset}} s^{-1} i_A\otimes s^{-1} (i_B;x)-s^{-1} (i_B;x)\otimes s^{-1}i_A 
\end{align}

The proof is by induction on $n$. We assume that  $i_{[n]}=(1,\ldots,n)$. Consequently  $s^{-1}i_{[n]}=\kappa_n$ and $s^{-1}(i_{[n]};x)=\gamma_n$.

Assume $n=2$. One has $d\kappa_1=d\mathfrak n_{0,1}=0$, hence (\ref{F1}) is proved. One has $d\gamma_1=d\mathfrak n_{1,1}=\mathfrak n_{0,2}\circ_1 \mathfrak n_{0,1}-\mathfrak n_{0,2}\cdot(21)\circ_1 \mathfrak n_{0,1}$ which writes $ds^{-1}(i_{[1]};x)=i_{[1]}\otimes (\emptyset;x)-(\emptyset;x)\otimes i_{[1]}$, proving (\ref{F2}).

Let $n>2$. Assume we have proven the second relation for $n-1$. The relation 
$\gamma_n=\gamma_1\circ^\op \gamma_{n-1}$ and $\gamma_1\circ^\op \mathfrak n_{0,2}=\mathfrak n_{0,2}\circ_2 \gamma_1\cdot (213)+\mathfrak n_{0,2}\circ_1\gamma_1$ gives
\[\begin{aligned}
ds^{-1}(i_{[n]};x)=&d(s^{-1}(i_{[1]};x))\circ_x s^{-1}(i_{\{2,\ldots,n\}};x)+\\
&s^{-1}(i_{[1]};x)\circ_x \sum_{A\sqcup B=\{2,\ldots,n\}} s^{-1} i_A\otimes s^{-1} (i_B;x)-s^{-1} (i_B;x)\otimes s^{-1}i_A\\
=&s^{-1}(i_{[n]})\otimes s^{-1}(\emptyset;x)-s^{-1}(\emptyset;x)\otimes s^{-1}(i_{[n]};x)+\\
&\sum_{A\sqcup B=\{2,\ldots,n\}} s^{-1} i_{\{1\}\cup A}\otimes s^{-1} (i_B;x)+s^{-1} i_A\otimes s^{-1} (i_{\{1\}\cup B};x)-\\
&\sum_{A\sqcup B=\{2,\ldots,n\}}s^{-1} (i_{\{1\}\cup B};x)\otimes s^{-1}i_A-
s^{-1} (i_B;x)\otimes s^{-1}i_{\{1\}\cup A}\\
=& \bigoplus_{A\sqcup B=[n]\atop{A\not=\emptyset}} s^{-1} i_A\otimes s^{-1} (i_B;x)-s^{-1} (i_B;x)\otimes s^{-1}i_A,
\end{aligned}\]
which proves (\ref{F2}).
Using $\kappa_n=\gamma_n\circ^\op \kappa_1$ one gets (\ref{F1}).

As a corollary, because $d^2=0$, $C$ is a  coassociative coalgebra where the coproduct
is given by

$$\begin{array}{cccl}
\Delta : & C & \rightarrow & C\otimes C \\
& i_{[n]} & \mapsto & \bigoplus\limits_{A\sqcup B=[n]} i_A\otimes i_B \\
& (i_{[n]}; x) & \mapsto & \bigoplus\limits_{A\sqcup B=[n]} i_A\otimes (i_B;x)+(i_B;x)\otimes i_A
\end{array}$$
and $\cal P=\Omega C=(T(s^{-1}\bar{C}),d)$ where $d$ is the unique derivation that lifts the coproduct on $C$.

\medskip

\noindent{\it Third step: $C$ is a Koszul coalgebra}. From the definition of $\Delta$, one sees that $C=D\otimes M$ where $D=T(X_\cl)$ is the unshuffle coalgebra and $M=U\oplus X_\op$ is the coalgebra with coproduct $\Delta(1)=1\otimes 1$ and $\Delta(x)=1\otimes x+x\otimes 1$. Following S. Priddy in \cite{Priddy70}, if $D$ and $M$ are Koszul coalgebras, so is $C$ and $C^{\ac}=D^{\ac}\otimes M^{\ac}$. From C. Stover in \cite{Stover93},
$D$ is  the enveloping algebra of the free Lie algebra $\Lie(X_\cl)$ which is a free cocommutative coalgebra cogenerated by $\Lie(X_\cl)$. But this is a Koszul coalgebra and its Koszul dual is $S(s^{-1}\Lie(X_\cl))$. Hence $D$ is Koszul. Following the notation of  \cite[section 3]{LodVal}, if $C(V,R)$ denotes the quadratic coalgebra cogenerated by $V$ with corelation $R$ then $C(V,R)^{\ac}=A(s^{-1}V,s^{-2}R)$ is the associative algebra generated by $s^{-1}V$ with relations $s^{-2}R$. Since $M=C(X_\op,0)$ then $M^{\ac}=A(s^{-1}X_\op,0)$ is the free associative algebra generated by $s^{-1}X_\op$, which is also Koszul. Consequently $C$ is Koszul and its dual Koszul algebra is

$$C^{\ac}(I,J)=S(s^{-1}\Lie(X_{\cl}))(I,\emptyset)\otimes T(s^{-1}X_\op)(\emptyset,J),$$

with the concatenation product.
The Koszul duality implies that the algebra map $\Omega C\rightarrow C^{\ac}$ is a quasi-isomorphism.

\medskip

\noindent{\it Fourth step}. Recall that the operad $\Lambda_\cl \OC$ is generated by 
$\tilde{\mathfrak n}_{0,2}$, $\tilde{\mathfrak l}_2$ of degrees $0$ and 
$\tilde{\mathfrak n}_{0,1}$ of degree $-1$ subject to the associativity relation for 
$\tilde{\mathfrak n}_{0,2}$, the Jacobi relation for  $\tilde{\mathfrak l}_2$, and the centrality relation 
for $\tilde{\mathfrak n}_{0,1}$. Any element in $\Lambda_\cl\OC([n],[l];\op)$ writes 
$\tilde\mu_p(l_{i_1},\ldots, l_{i_p})\otimes (\tilde\mu_{l}\cdot \tau)$ with

\begin{itemize}
\item
$l_{i_j}=\tilde{\mathfrak n}_{0,1}\circ \tilde{\mathfrak l}_{i_j}\in s^{-1}\Lie(I_j)$ where $\tilde{\mathfrak l}_{i_j}$ is an iterated composition of $\tilde{\mathfrak l}_2$ combined with permutations, which is called a Lie element in  $\Lambda_\cl \OC$;
\item $\sqcup_{1\leq j\leq k} I_j=[n]$;
\item $\tau\in S_l$.
\end{itemize}
Furthermore, the centrality of $\tilde{\mathfrak n}_{0,1}$ implies that  the product $\tilde\mu_p$ is invariant under the action of the symmetric group. The algebra structure follows and one has 
$\Lambda_\cl\OC=C^{\ac}$ as an algebra.
Consequently
$$\psi:H_0(\SC)^{\ac}=\Omega C\rightarrow C^{\ac}=\Lambda_\cl \OC$$
is a quasi-isomorphism of algebras and $H_*(\psi):H_*(\cal P)\rightarrow \Lambda_\cl \OC$ is an isomorphism.

The map $\phi:\Lambda_\cl \OC\rightarrow H_*(\cal P)$ is an algebra map. It suffices to prove that
$H_*(\psi)\circ\phi$ is the identity on a Lie element of type 
$\tilde{\mathfrak n}_{0,1}\circ\tilde{\mathfrak l}_{i_j}$. Because of the eye relation in $\cal P$ this Lie element is sent exactly to a Lie element in $\Lie(X_\cl)\subset \Omega s^{-1}S^{c}\Lie(X_\cl)$, hence to $\tilde{\mathfrak n}_{0,1}\circ\tilde{\mathfrak l}_{i_j}$ via $H_*(\psi)$.

Finally, $\phi$ is an isomorphism of algebras. \end{proof}

From the above proof one gets the following corollary.

\begin{cor}\label{C:isoalg} The operads $\OC_\infty$, $H_0(\SC)^!$ and $\Lambda_{\cl}\OC$ are multiplicative operads, hence are associative algebras in the category of  $2$-collections. The maps $\OC_\infty\rightarrow H_0(\SC)^!\rightarrow \Lambda_{\cl}\OC$ are quasi-isomorphisms of associative algebras.
 \end{cor}

As an application of what we have done, we can show that the OCHA operad $\cal{OC}_\infty$ is non-formal. In particular there is no hope for the second map defined in the previous corollary to be a quasi-isomorphism of operads. 

\begin{prop}\label{P:nonformality}
 The OCHA operad $\cal{OC}_\infty$ is non-formal.
\end{prop}
\begin{proof}
Since $\cal{OC}_\infty$ is a minimal operad whose homology is $\Lambda_\cl\cal{OC}$, 
its formality would imply the existence of an operad morphism 
$\cal{OC}_\infty \to \Lambda_c\cal{OC}$ inducing an isomorphism in homology. Note that $\mathfrak n_{1,1} \in \cal{OC}_\infty$ is an element 
of degree $0$. But all elements in $\Lambda_\cl\cal{OC}(1,1;\op)$ have degree $-1$ because they can only be obtained by composing 
$\tilde{\mathfrak n}_{1,0}$ and $\tilde{\mathfrak n}_{0,2}$. Hence, any quasi-isomorphism 
$\cal{OC}_\infty \to \cal{OC}$ would take $\mathfrak n_{1,1}$ to $0$. On the other hand, 
since $\mathfrak l_2$ and $\mathfrak n_{1,0}$ are generators of the homology of $\cal{OC}_\infty$ they cannot be taken to zero by any quasi-isomorphism. 
The ``Eye Law'' says that $\mathfrak n_{1,0} \circ_1 \mathfrak l_2 =\mathfrak n_{1,1} \circ_2\mathfrak n_{1,0} + (\mathfrak n_{1,1} \circ_2 \mathfrak n_{1,0})\cdot (21)$, a contradiction. 
\end{proof}

\bibliographystyle{amsplain}
\bibliography{bibliomars2011}
\bigskip
\end{document}